\documentclass[11pt]{article}
\usepackage{amssymb}
\usepackage{graphicx}
\usepackage{setspace}
  \usepackage{paralist}
  \usepackage{longtable}
   \usepackage{multirow}
    \usepackage{rotating}

\usepackage{fancyhdr}
\usepackage[pagewise]{lineno}

\usepackage{amsmath, amsthm, amssymb}
\usepackage{authblk}
\usepackage{graphicx}
\usepackage{float}
\usepackage{hyperref}
\usepackage[margin=1in]{geometry}
\usepackage{comment}
\allowdisplaybreaks[4]
\numberwithin{equation}{section}
\date{}

\newtheorem{theorem}{Theorem}[section]
\newtheorem{proposition}{Proposition}[section]
\newtheorem{lemma}{Lemma}[section]
\newtheorem{corollary}{Corollary}[section]

\newtheorem{remark}{Remark}[section]

\newcommand{\be}{\begin{equation}}
\newcommand\ee{\end{equation}}
\newcommand\bes{\begin{eqnarray}}
\newcommand\ees{\end{eqnarray}}
\newcommand\bess{\begin{eqnarray*}}
\newcommand\eess{\end{eqnarray*}}

\title{\bf Wave propagation of a non-conservative compressible two-fluid model with different pressures}

\author{
Zhigang Wu$^{1}$,\ 
Yinghui Zhang$^{1}$\thanks{Corresponding author (Y.H. Zhang): yinghuizhang@mailbox.gxnu.edu.cn}}

\begin{document}

\maketitle
\renewcommand{\thefootnote}{\fnsymbol{footnote}}

\footnotetext{1. School of Mathematics and Statistics, Guangxi Normal University, Guilin,
 Guangxi 541004, P.R. China}

\pagestyle{myheadings} \markboth{}{Z. WU \& Y. Zhang:\ \ \  Two-phase Model } \maketitle


 {{\bf  Abstract:}} 
This paper investigates the wave propagation for a non-conservative compressible two-fluid model with unequal pressures in $\mathbb{R}^3$. This model is novel and fundamentally distinct from the classical single-phase compressible Navier–Stokes equations proposed by Liu–Wang [28]. This work fills a significant gap, as prior results largely apply only to conservative or partially conservative systems with specific Green's function cancellations. The non-conservative nature and the coexistence of two distinct-speed Huygens waves demand dedicated nonlinear coupling analysis and add further complexity. Our study reveals fundamentally different physical behaviors of this non-conservative two-phase model under equal and unequal interfacial pressures. For equal pressure, the model admits a stationary diffusion wave comprising a single Huygens wave, analogous to classical Navier–Stokes equations. In sharp contrast, under unequal pressure, the stationary diffusion wave consists of two Huygens waves propagating at different speeds—a novel phenomenon absent in the classical Navier–Stokes context.
 \bigskip


\bigskip


\section{\leftline {\bf{Introduction.}}}
\setcounter{equation}{0}
\subsection{Background and motivation}
\indent As is well--known, most of the flows encountered in nature are
multifluid flows. They are widely used in nuclear power, chemical
processing, oil and gas manufacturing, and so on. To address the inherent complexity of multiphase flows and meet the practical modeling demands of engineers, the volume-averaging method has emerged as a well-established and widely adopted simplification approach. This methodological framework gives rise to the so-called averaged multiphase models, whose detailed derivation and theoretical foundations can be found in relevant studies, cf. \cite{Bear, Brennen1, Raja}. Through the systematic implementation of this volume-averaging procedure, a generic form of the compressible two-phase fluid model can be derived as follows:
\begin{equation}\label{1.1}
\left\{\begin{array}{l}
\alpha^{+}+\alpha^{-}=1, \\
\partial_{t}\left(\alpha^{\pm} \rho^{\pm}\right)+\operatorname{div}\left(\alpha^{\pm} \rho^{\pm} \mathbf{u}^{\pm}\right)=0, \\
\partial_{t}\left(\alpha^{\pm} \rho^{\pm} \mathbf{u}^{\pm}\right)+\operatorname{div}\left(\alpha^{\pm} \rho^{\pm} \mathbf{u}^{\pm} \otimes \mathbf{u}^{\pm}\right)
+\alpha^{\pm} \nabla P^{\pm}\left(\rho^{\pm}\right)=\operatorname{div}\left(\alpha^{\pm} \tau^{\pm}\right), \\
P^{+}\left(\rho^{+}\right)-P^{-}\left(\rho^{-}\right)=f\left(\alpha^{-}
\rho^{-}\right),
\end{array}\right.
\end{equation}
where $\rho^{\pm}(x, t) \geq 0, \mathbf{u}^{\pm}(x, t)$ and
$P^{\pm}\left(\rho^{\pm}\right)=A^{\pm}\left(\rho^{\pm}\right)^{\bar{\gamma}^{\pm}}$
denote the densities, the velocity of each phase, and the two
pressure functions, respectively. The variable $0 \leq \alpha^{+}(x, t) \leq 1$ is the
volume fraction of fluid $+$ (the liquid), $0 \leqq
\alpha^{-}(x, t) \leq 1$ is that of fluid $-$ (the gas),
 the viscous stress tensors are given by
$\tau^{\pm} = \mu^{\pm}(\nabla \mathbf{u}^{\pm} + \nabla^{T} \mathbf{u}^{\pm}) + \lambda^{\pm} (\operatorname{div} \mathbf{u}^{\pm}) \mathbb{I}$, and $\bar{\gamma}^{\pm} \geq 1,
A^{\pm}>0$ are positive constants. In what follows, we set
$A^{+}=A^{-}=1$ without loss of any generality.

When \( f\equiv0 \) (or \( P^+=P^- \)) in model (\ref{1.1}), considerable research has focused on the global existence and large time behavior of weak, strong, and classical solutions. For weak solutions, Bretsch et al. \cite{Bresch1} studied the 3D periodic domain with internal capillary forces, while Bresch--Huang--Li \cite{Bresch2} investigated the 1D whole space without such forces. Regarding strong and classical solutions under small perturbation, existing studies fall into two categories based on the presence or absence of capillary forces.
For model (\ref{1.1}) with capillary forces, Cui et al. \cite{Cui} derived the time-decay rates of classical solutions under the constraints of special density-dependent viscosity coefficients and equal capillary coefficients. This result was later extended by \cite{Liy}, where the authors established decay rates for the same model with general constant viscosities and capillary coefficients, relying on spectral analysis and energy estimates.
In contrast, model (\ref{1.1}) without capillary forces presents inherent challenges: the Green's matrix possesses a zero eigenvalue, which complicates the derivation of energy estimates necessary for proving global existence. Only recently have breakthroughs been made in this direction: Wu-Yao-Zhang \cite{Wug1} and Shou et al. \cite{Shou} established the global existence and temporal decay rates of strong (or smooth) solutions in the Sobolev and critical Besov frameworks, respectively.

Progress has also been made regarding the well-posedness of solutions when \( f\neq0 \) (\( P^+\neq P^- \)). In their influential work \cite{Evje9}, Evje--Wang--Wen proved the global stability of the constant equilibrium state for the 3D Cauchy problem of model (\ref{1.1}), under the assumptions that the initial perturbation is small in \( H^2 \)-norm and bounded in \( L^1 \)-norm, along with the stability condition \( f'<0 \). Conversely, Wu-Yao-Zhang \cite{Wug2} demonstrated the instability of solutions when \( f'>0 \). For more information about
the model (\ref{1.1}) and related models, we refer to \cite{
Evje3, Evje4, Evje8, Friis1, Ishii1,Prosperetti,Raja,Vasseur, Wen1, Yao2,
Zhang4} and references therein.

The main purpose of this work is to establish the large-time wave propagation structure of classical solutions to the non-conservative two-fluid model \eqref{1.1} with $P^+\neq P^-$. This study mainly extends our recent work \cite{Wu7}, where the wave propagation behavior of model \eqref{1.1} with \( P^+=P^- \) exhibits the {\it generalized Huygens principle}—a property first established for one-phase compressible Navier–Stokes equations \cite{Lw}. This principle dictates that the time-asymptotic profile of the solution decomposes into a superposition of a stationary diffusion wave (D-wave) of the form \( (1+t)^{-\frac{3}{2}}\big(1+\frac{|x|^2}{1+t}\big)^{-\frac{3}{2}} \), and a moving diffusion wave (H-wave) given by \( (1+t)^{-2}\big(1+\frac{(|x|-{\rm c}t)^2}{1+t}\big)^{-\frac{3}{2}} \), with the propagation speed \( {\rm c} \) specified in Theorem \ref{l 1.1}. The H-wave arises from wave-type components in the low-frequency region of Green's function; since it decays more slowly than the D-wave in \( L^p(\mathbb{R}^3) \) for \( p<2 \), and hence its nonlinear coupling requires extracting one derivative from nonlinear terms— a step typically dependent on the conservative (divergence) structure of the equations, as shown in earlier studies \cite{Ls,Lw}. We would like to refer to other related results for the long-time behaviors of the one-phase model in \cite{Dan1,Duan,Duan3,Guo2,Hoff1,Hoff2, Kobayashi2,Lihl,L7,Mat1,Mat2,Wang-Tan,Zeng1} and the kinetic equations in \cite{Lihl2,Lihl3,L4,L6}.

Compared to classical one-phase models (i.e., compressible Navier-Stokes equations \cite{Ls, Lw}) and the model (\ref{1.1}) with  \(P^+= P^-\) \cite{Wu7}, the analysis of the two-fluid system (\ref{1.1}) with \( P^+\neq P^- \) confronts two key structural obstacles: First, the system is inherently non-conservative due to the pressure terms \( \alpha^\pm\nabla P \) in the momentum equations, which preclude the direct application of cancellation and reformulation techniques that are critical to the analysis of conservative frameworks (e.g., \cite{Ls, Lw}). Second, linear analysis of model (\ref{1.1}) reveals two Huygens waves with different propagation speeds—a feature entirely absent in the \( P^+=P^- \) case \cite{Wu7}, where the system obeys the typical generalized Huygens principle: a stationary diffusion wave (D-wave) superimposed on a single moving diffusion wave (H-wave), basically consistent with the wave behavior of classical one-phase Navier-Stokes equations \cite{Ls, Lw}.

The key innovation of this paper consists in resolving the simultaneous interaction between Huygens waves of distinct propagation speeds within a non-conservative framework. In fact, the conservative structure is rather important in \cite{Ls, Lw} when studying the interaction of the Huygens waves with equal propagation speed:
\begin{equation*}\label{0.1}
\begin{array}{ll}
\! &\displaystyle\int_{0}^{t}\!\!\int_{\mathbb{R}^3}\underline{(1+t-\tau)^{-\frac{5}{2}}}\Big(1+\frac{(|x-y|-{\rm c}(t-\tau))^2}{1\!+\!t\!-\!\tau}\Big)^{-N}(1+\tau)^{-4}\Big(1+\frac{(|y|-{\rm c}\tau)^2}{1\!+\!\tau}\Big)^{-3}dyd\tau\\[2mm]
\ \ \ \ \lesssim\!\!\!\! \ &\displaystyle(1+t)^{-2}\Big(\big(1+\frac{|x|^2}{1+t}\big)^{-\frac{3}{2}}+\Big(1+\frac{(|x|-{\rm c}t)^2}{1+t}\Big)^{-\frac{3}{2}}\Big),\ \ \ \ \ \ \ N\ {\rm is\ suitably\ large}.
\end{array}
\end{equation*}
Here, the additional $(1+t)^{-\frac{1}{2}}$-decay rate (underlined) typically arises from the system’s conservative (divergence) form—an advantage our non-conservative model entirely lacks. This absence poses a critical difficulty: the H-wave $(1+t)^{-2}\big(1+\frac{(|x|-{\rm c}t)^2}{1+t}\big)^{-N}$ decays more slowly than the heat kernel in $L^p$ spaces for $1\leq p<2$; even worse, its $L^p$-estimates grow in time when $p<\frac{3}{2}$, making nonlinear coupling analysis far more challenging.
On the other hand, Bai-Zhang \cite{Bai} identified two H-waves with different propagation speeds for conservative compressible viscoelastic fluids, and verified the generalized Huygens principle—a D-wave superimposed on two distinct H-waves—within a conservative framework. The key challenge in our work lies in establishing the necessary pointwise estimates for the interaction of these distinct H-waves under non-conservative conditions. Unlike \cite{Bai}, which relies on the convenience of a conservative structure, our model’s non-conservative nature brings new difficulties that render conventional analytical methods ineffective. To address these new difficulties arising from the non-conservative setting (a key difference from \cite{Bai}), we revisit nonlinear convolution estimates through refined space-time decomposition techniques, combined with full exploitation of the interchange between temporal and spatial variables in specific regions. For example, when handling the domain $D_{22}:=\{\frac{{\rm c}_1t}{2}\leq|x|\leq{\rm c}_1t-\sqrt{1+t}\}$ in Lemma \ref{l 4.2}, we split it into four parts. In the subcase where the temporal integral interval satisfies $\frac{{\rm c}_1t-|x|}{2({\rm c}_1+{\rm c}_2)}\leq\tau\leq \frac{t}{2}$ (with propagation speeds ${\rm c}_1$ and ${\rm c}_2$), adopting the approach in \cite{Bai} would lead to a logarithmic growth factor, making it impossible to close the pointwise space-time ansatz for the nonlinear problem. Our key innovation here is to tackle this subcase directly using spherical coordinates, thereby obtaining sharper space-time estimates from the convolution. Details of the other cases are provided in the proof of nonlinear coupling in Section 5.2.

\subsection{New formulation of system \eqref{1.1} and Main Results}
In this subsection, we devote ourselves to reformulating the system
\eqref{1.1} and stating the main results. The relations between the
pressures of $\eqref{1.1}_3$ implies
\begin{equation}\label{1.9}
{\rm d}P
  =
  s_+^2
 {\rm d}\rho^+
  =
   s_-^2
    {\rm d}\rho^-,
 \quad
  {\rm where}
   \quad
    s_\pm
     :=
      \sqrt{ \frac{{\rm d}P}{{\rm d}\rho^\pm}(\rho^\pm)}.
\end{equation}
Here $s_\pm$ represent the sound speed of each phase respectively.
As in \cite{Bresch1}, we introduce the fraction densities
\begin{equation}\label{1.10}
R^\pm
 =
  \alpha^\pm \rho^\pm,
\end{equation}
which together with the relation $\alpha^++\alpha^-=1$ leads to
\begin{equation}\label{1.11}
{\rm d}
 \rho^+
  =
   \frac{1}{\alpha_+}
    ({\rm d}R^+
       -
        \rho^+
         {\rm d}\alpha^+),
  \quad
   {\rm d}
    \rho^-
     =
      \frac{1}{\alpha_-}
       ({\rm d}R^-
         +
          \rho^-
           {\rm d}\alpha^+).
\end{equation}
From \eqref{1.9}--\eqref{1.10}, we finally get
\begin{equation}\notag
{\rm d}\alpha^+
 =
  \frac{\alpha^-s_+^2}{\alpha^-\rho^+s_+^2+\alpha^+\rho^-s_-^2}
   {\rm d}R^+
    -
     \frac{\alpha^+s_-^2}{\alpha^-\rho^+s_+^2+\alpha^+\rho^-s_-^2}
   {\rm d}R^-.
\end{equation}
Substituting the above equality into \eqref{1.11} yields
\[
\mathrm{d} \rho^{+}=\frac{\rho^{+} \rho^{-}
s_{-}^{2}}{R^{-}\left(\rho^{+}\right)^{2}
s_{+}^{2}+R^{+}\left(\rho^{-}\right)^{2} s_{-}^{2}}\left(\rho^{-}
\mathrm{d} R^{+}+\left(\rho^{+}+\rho^{+} \frac{\alpha^{-}
f^{\prime}}{s_{-}^{2}}\right) \mathrm{d} R^{-}\right),
\]
and
\[
\mathrm{d} \rho^{-}=\frac{\rho^{+} \rho^{-}
s_{+}^{2}}{R^{-}\left(\rho^{+}\right)^{2}
s_{+}^{2}+R^{+}\left(\rho^{-}\right)^{2} s_{-}^{2}}\left(\rho^{-}
\mathrm{d} R^{+}+\left(\rho^{+}-\rho^{-} \frac{\alpha^{+}
f^{\prime}}{s_{+}^{2}}\right) \mathrm{d} R^{-}\right),
\]
which combined with $(\ref{1.9})$ imply for the pressure differential
\[
\mathrm{d} P^{+}=\mathcal{C}^{2}\left(\rho^{-} \mathrm{d}
R^{+}+\left(\rho^{+}+\rho^{+} \frac{\alpha^{-}
f^{\prime}}{s_{-}^{2}}\right) \mathrm{d} R^{-}\right) ,\] and
\[
\mathrm{d} P^{-}=\mathcal{C}^{2}\left(\rho^{-} \mathrm{d}
R^{+}+\left(\rho^{+}-\rho^{-} \frac{\alpha^{+}
f^{\prime}}{s_{+}^{2}}\right) \mathrm{d} R^{-}\right) ,\] where
\[
\mathcal{C}^{2}:=\frac{s_{-}^{2} s_{+}^{2}}{\alpha^{-} \rho^{+}
s_{+}^{2}+\alpha^{+} \rho^{-} s_{-}^{2}}.\]
Next, by using the relation $\alpha^+ +\alpha^- =1$ again, we have
\begin{equation}\label{1.13}
\frac{R^+}{\rho^+}
 +
  \frac{R^-}{\rho^-}
   =1,
    \quad
     {\rm and \ therefore}
      \quad
       \rho^-
        =
         \frac{R^-\rho^+}{\rho^+-R^+}.
\end{equation}
By virtue of $\eqref{1.1}_3$, we have
\begin{equation}\notag
\varphi(\rho^+, R^+, R^-)
 :=P(\rho^+)-P\left(\frac{R^-\rho^+}{\rho^+-R^+}\right)-f(R^-)=0.
\end{equation}
Consequently, for any given two positive constants $\tilde R^+$,
$\tilde R^-$, there exists $\tilde \rho^+>\tilde R^+$ such that
\begin{equation}\notag
\varphi(\tilde \rho^+, \tilde R^+, \tilde R^-)=0.
\end{equation}
Differentiating $\varphi$ with respect to $\rho^+$, we get
\begin{equation}\notag
\frac{\partial\varphi}{\partial\rho^+}(\rho^+, R^+, R^-)=s_+^2+s_-^2\frac{R^-R^+}{(\rho^+-R^+)^2}>0,
\end{equation}
which together with Implicit Function Theorem and \eqref{1.10},
\eqref{1.13} implies that the unknowns $\rho^\pm$, $\alpha^\pm$ and
$\mathcal{C}$ can be given by
\begin{equation}\notag
\rho^\pm
     =
      \varrho^\pm(R^+,R^-),
\qquad
      \alpha^\pm
       =
        \alpha^\pm(R^+,R^-),
         \quad
     {\rm and \ therefore}
      \quad
      \mathcal{C}=\mathcal{C}(R^+, R^-).
\end{equation}
 We refer to [\cite{Bresch1}, pp. 614] for the details.
\par

Notice that the non-conservation of system (\ref{1.1}) is mainly from two pressure terms $\alpha^\pm\nabla P^\pm$ in the nonlinear part, since the other nonlinear terms are of the divergence form. As mentioned above, the conservation is critical when deducing the generalized Huygens' principle. To this end, we still need to use the divergence form of the other nonlinear terms when dealing with the nonlinear coupling. Therefore, we consider the following system with respect to the unknowns $(R^\pm,\ \mathbf{m}^\pm=R^\pm \mathbf{u}^\pm)$:
\begin{equation}\label{1.14}
\left\{\begin{array}{l}
\partial_{t} R^{\pm}+\operatorname{div}\mathbf{m}^\pm=0, \\
\partial_{t}\mathbf{m}^++\operatorname{div}\big(\frac{\mathbf{m}^+\otimes\ \!\mathbf{m}^+}{R^+}\big)+\alpha^{+} \mathcal{C}^{2}\big[\rho^{-} \nabla
R^{+}+\rho^{+}\nabla R^{-}\big] \\
\hspace{1.8cm}=\operatorname{div}\big\{\alpha^{+}\big[\mu^{+}\big(\nabla
\frac{\mathbf{m}^+}{R^+}+\nabla^{T} \frac{\mathbf{m}^+}{R^+}\big)
+\lambda^{+} \operatorname{div} \frac{\mathbf{m}^+}{R^+} \mathbb{I}\big]\big\}, \\
\partial_{t}\mathbf{m}^-+\operatorname{div}\big(\frac{\mathbf{m}^-\otimes\ \!\mathbf{m}^-}{R^-}\big)+\alpha^{-} \mathcal{C}^{2}\big[\rho^{-} \nabla
R^{+}+\rho^{+}\nabla R^{-}\big] \\
\hspace{1.8cm}=\operatorname{div}\big\{\alpha^{-}\big[\mu^{-}\big(\nabla
\frac{\mathbf{m}^-}{R^+}+\nabla^{T} \frac{\mathbf{m}^-}{R^-}\big)
+\lambda^{-} \operatorname{div} \frac{\mathbf{m}^-}{R^-} \mathbb{I}\big]\big\},
\end{array}\right.
\end{equation}
subject to the initial condition
\begin{equation}\label{1.15} (R^{+}, \mathbf{m}^{+}, R^-, \mathbf{m}^{-})(x,
0)=(R_{0}^{+}, \mathbf{m}_{0}^{+}, R_{0}^+, \mathbf{m}_{0}^{-})(x)\rightarrow(\bar
R^{+}, \mathbf{0}, \bar R^{-}, \mathbf{0}) \quad
\hbox{as}\quad |x|\rightarrow\infty \in \mathbb{R}^{3},
\end{equation}
where the constants  $\bar{R}^{+}$ and $\bar{R}^{-}$ denote
the background doping profile and are taken as
1 for simplicity in the present paper.

\medskip
Now, we are in a position to state our main result.
\smallskip
\begin{theorem}\label{l 1.1} Suppose that the following stability condition holds
\begin{equation}\label{1.12}
-\frac{s_{-}^{2}(1,1)}{\alpha^{-}(1,1)}<f^{\prime}(1)<\frac{\eta-s_{-}^{2}(1,1)}{\alpha^{-}(1,1)}<0,
\end{equation}
where $\eta$ is a positive, small fixed constant. When the initial data $(R_{0}^{\pm}-1,\ \mathbf{m}_{0}^{\pm})\in H^5(\mathbb{R}^3)$, and has the compact support or satisfies the pointwise assumption for $|\alpha|\leq 1$ that
\begin{equation}\label{1.16}
|\partial_x^\alpha(R_0^\pm-1,\mathbf{m}_0^\pm)|\leq \varepsilon_0(1+|x|^2)^{-r},\ \ r>\frac{19}{10},
\end{equation}
then the Cauchy problem \eqref{1.14}--\eqref{1.15} admits a unique
solution $\left(R^{+}, \mathbf{m}^{+}, R^{-}, \mathbf{m}^{-}\right)$ globally in time
and obeys the following pointwise space-time descriptions
\begin{equation}\label{1.17}
|R^\pm-1|\lesssim (1+t)^{-2}\Big(1+\frac{|x|^2}{1+t}\Big)^{-1}+(1+t)^{-2}\Big(1+\frac{(|x|\!-\!{\rm c}_1t)^2}{1+t}\Big)^{-1}+(1+t)^{-2}\Big(1+\frac{(|x|\!-\!{\rm c}_2t)^2}{1+t}\Big)^{-1},
\end{equation}
\begin{equation}\label{1.18}
|\mathbf{m}^\pm|\lesssim (1+t)^{-\frac{3}{2}}\Big(1+\frac{|x|^2}{1+t}\Big)^{-\frac{3}{2}}+(1+t)^{-2}\Big(1+\frac{(|x|\!-\!{\rm c}_1t)^2}{1+t}\Big)^{-1}+(1+t)^{-2}\Big(1+\frac{(|x|\!-\!{\rm c}_2t)^2}{1+t}\Big)^{-1},
\end{equation}
where the propagation speeds ${\rm c}_1$ and ${\rm c}_2$ are
\begin{equation}\label{1.19}
{\rm c}_1=\sqrt{\frac{\beta_1+\beta_4}{2}-\sqrt{\frac{(\beta_1+\beta_4)^2}{4}-(\beta_1\beta_4-\beta_2\beta_3)}},
\end{equation}
\begin{equation}\label{1.20}
{\rm c}_2=\sqrt{\frac{\beta_1+\beta_4}{2}+\sqrt{\frac{(\beta_1+\beta_4)^2}{4}-(\beta_1\beta_4-\beta_2\beta_3)}},
\end{equation}
and $\beta_1,\beta_2,\beta_3,\beta_4$ are defined in (\ref{1.23}).
\end{theorem}

\begin{remark}Notice that ${\rm c}_1\neq{\rm c}_2$ due to the fact $\beta_1\beta_4>\beta_2\beta_3$ according to (\ref{1.23}) and the stability condition (\ref{1.12}). This key distinction allows us to reveal a critical difference between our two-fluid model with unequal pressures and the equal-pressure two-fluid model studied in \cite{Wu7}, where only one Huygens wave emerges in the pointwise description of the solution, whereas our model gives rise to two distinct Huygens waves. This finding constitutes a novel observation that advances the understanding of two-fluid dynamics. Additionally, the pointwise behavior of our two-fluid model differs fundamentally from that of the standard one-fluid model, specifically the compressible Navier-Stokes equations established in \cite{Ls,Lw}.
\end{remark}

\begin{remark}The Huygens wave in Theorem \ref{l 1.1} decays as $(1+t)^{-2}\big(1+\frac{(|x|-{\rm c}t)^2}{1+t}\big)^{-1}$, which is notably weaker than $(1+t)^{-2}\big(1+\frac{(|x|-{\rm c}t)^2}{1+t}\big)^{-\frac{3}{2}}$ reported for conservative systems—including the compressible Navier-Stokes equations \cite{Ls,Lw} and their bipolar Poisson-coupled variants \cite{Wu4,Wu5}. This weakened spatial decay is a direct result of the non-conservative terms in our model, presenting a major technical challenge: unlike our prior work on other fluid models, we cannot rely on cancellations in the low-frequency component of Green's function to offset the non-conservative structure of the two-fluid system. This difficulty required the development of new analytical strategies to establish the desired pointwise estimates.
\end{remark}

\begin{remark}\label{r 1.3} Another difference from the single-phase flow (compressible Navier–Stokes system) concerns the  $L^p$-estimates of momenta. In the low-frequency regime, the Riesz operator singularity in the two-fluid Green's function cannot be fully canceled, unlike the effective cancellation of the single-phase flow in \cite{Hoff1,Hoff2}. This leaves an irreducible residual ``Riesz wave I\!I\!I" in (\ref{3.7})-(\ref{3.8}), which prevents an $L^1$-estimate for the momenta, while $L^p (p>1)$-estimate of momenta remains valid. Specifically, according to the linear analysis in Section 3 and standard $L^p-L^q$ estimation method, we have
\begin{align}\label{1.22}
&\|R^\pm-1\|_{L^p(\mathbb{R}^3)}\lesssim \Bigg\{\!\!\begin{array}{cc}
  \displaystyle (1+t)^{-(2-\frac{5}{2p})}, & 1\leq p\leq2, \\
   \displaystyle (1+t)^{-(\frac{3}{2}-\frac{3}{2p})}, & 2\leq p\leq\infty,
\end{array}
\end{align}
\vspace{-3mm}
\begin{align}
&\|\mathbf{m}^\pm\|_{L^p(\mathbb{R}^3)}\lesssim \Bigg\{\!\!\begin{array}{cc}
  \displaystyle (1+t)^{-(2-\frac{5}{2p})}, & 1<p\leq2, \\
   \displaystyle (1+t)^{-(\frac{3}{2}-\frac{3}{2p})}, & 2\leq p\leq\infty,
\end{array}
\end{align}
if the initial perturbation is suitably small in $H^3(\mathbb{R}^3)\cap L^1(\mathbb{R}^3)$.


\end{remark}


\smallskip
\smallskip

\indent Now, let us outline the proof strategy for Theorem \ref{l 1.1} and highlight the main difficulties and key techniques.

First, due to the presence of two pressure terms, the original model does not possess a conservative structure. To more clearly illustrate the key technical difficulties arising from the non-conservative nonlinear pressure terms (in non-divergence form) that complicate subsequent nonlinear analysis, we choose the standard conserved quantities—mass and momentum—as the unknown variables for our model study. Specifically, we choose the unknowns $(R^\pm,\ \mathbf{m}^\pm=R^\pm \mathbf{u}^\pm)$ instead of the $(R^\pm,\ \mathbf{u}^\pm)$ employed by Wang $et.\ al.$ \cite{Wangh}. This deliberate variable transformation is non-trivial, as it fundamentally alters the structure of the nonlinear terms in the model, laying the foundation for addressing the non-conservative challenges inherent to our system.
To see this, by setting $n^{\pm}=R^{\pm}-1$, we can rewrite the linearized system of model \eqref{1.14} in terms of the variables $(n^+, \mathbf{m}^+, n^-, \mathbf{m}^-)$ as follows:
\begin{equation}\label{1.27}
\left\{\begin{array}{l}
\partial_{t} n^++\operatorname{div}\mathbf{m}^+=0, \\
\partial_{t}\mathbf{m}^{+}+\beta_1\nabla n^++\beta_2\nabla
n^--\nu^+_1\Delta \mathbf{m}^+-\nu^+_2\nabla\operatorname{div} \mathbf{m}^+=F_1, \\
\partial_{t} n^-+\operatorname{div}\mathbf{m}^-=0, \\
\partial_{t}\mathbf{m}^{-}+\beta_3\nabla n^++\beta_4\nabla
n^--\nu^-_1\Delta \mathbf{m}^--\nu^-_2\nabla\operatorname{div} \mathbf{m}^-=F_2, \\
\end{array}\right.
\end{equation}
where $\nu_{1}^{\pm}=\frac{\mu^{\pm}}{\bar{\rho}^{\pm}}$,
$\nu_{2}^{\pm}=\frac{\mu^{\pm}+\lambda^{\pm}}{\bar{\rho}^{\pm}}>0$, and the coefficients $\beta_i$ ($i=1,2,3,4$), which are closely related to the pressure terms, are given by
\begin{eqnarray}
\beta_{1}=\frac{\mathcal{C}^{2}(1,1)
\rho^{-}(1,1)}{\rho^{+}(1,1)},\ \
\beta_{2}=\mathcal{C}^{2}(1,1)+\frac{\mathcal{C}^{2}(1,1)
\alpha^{-}(1,1) f^{\prime}(1)}{s_{-}^{2}(1,1)},\nonumber\\
\beta_{3}=\mathcal{C}^{2}(1,1),\ \
\beta_{4}=\frac{\mathcal{C}^{2}(1,1)
\rho^{+}(1,1)}{\rho^{-}(1,1)}-\frac{\mathcal{C}^{2}(1,1)
\alpha^{+}(1,1) f^{\prime}(1)}{s_{+}^{2}(1,1)}.\label{1.23}
\end{eqnarray}
A further key step in addressing the non-conservative difficulties is our decomposition of the nonlinear terms $F_1$ and $F_2$, which we rewrite as
\begin{equation}\label{1.29}
F_1=\tilde{F}_1-\underbrace{[\alpha^+\nabla P-(\beta_1\nabla n^++\beta_2\nabla n^-)]}_{Q_1}
,\ {\rm and}\ \ F_2=\tilde{F}_2-\underbrace{[\alpha^-\nabla P-(\beta_3\nabla n^++\beta_4\nabla n^-)]}_{Q_2},
\end{equation}
where $\tilde{F}_1$ and $\tilde{F}_2$ represent the divergence-form terms in $F_1$ and $F_2$—these terms are well-behaved and do not introduce new technical obstacles. In contrast, the non-divergence-form components $Q_1$ and $Q_2$ are the primary source of difficulty, as they embody the non-conservative nature of the pressure terms that distinguish our work from existing studies (e.g., \cite{Wangh}). The detailed expressions of $F_1$ and $F_2$, including the specific form of the challenging  terms $Q_1$ and $Q_2$, will be provided in Section 2, laying the groundwork for our subsequent nonlinear analysis.

Second, we perform a spectral analysis based on low-middle-high-frequency decomposition. The main difficulty and difference lie in low frequencies compared with one-phase model (compressible Navier-Stokes equations).  On one hand, the Hodge decomposition introduces nonlocal Riesz operators with symbols $\frac{\xi}{|\xi|}$ or $\frac{\xi\xi^T}{|\xi|^2}$—these are formally singular in the low-frequency regime. When combined with the wave operators and heat operators appearing in the Green's function, these singularities create a major obstacle. In fact, for one-phase model (compressible Navier-Stokes equations in Hoff and Zumbrun \cite{Hoff1,Hoff2}, they can completely remove the singularity by establishing an ``effective artificial viscosity" system and using the cancellation of the compressible acoustic component \(G_2 = w_t * K_{\mu+\lambda} * R\) and the incompressible shear component \(G_3 = K_\mu * (\delta \mathbb{I} - R)\) exhibit a low-frequency cancellation in Fourier space: $\hat G_2 + \hat G_3 = e^{-\mu|\xi|^2 t} \mathbb{I} + \bigl(\cos(c|\xi|t) e^{-(\mu+\lambda)|\xi|^2 t} - e^{-\mu|\xi|^2 t}\bigr) \frac{\xi\xi^T}{|\xi|^2}$, where $K_{\mu}$ and $K_{\mu+\lambda}$ are the heat kernels, $R$ denotes the double Riesz operator, and the coefficient of the singular Riesz projector \(\frac{\xi\xi^T}{|\xi|^2}\) vanishes to order \(O(|\xi|^2)\) at \(\xi=0\), thereby eliminating the non-integrable \(|x|^{-n}\) tail and upgrading the decay to \(|x|^{-(n+2)}\), so that the combined kernel \((G_2+G_3)(\cdot,t)\) belongs to \(L^1(\mathbb{R}^n)\) (up to a logarithmic correction when \(n=2\)). However, for two-fluid model (\ref{1.1}), one cannot fully remove this singularity in the low frequency of Green's function due to the complicity and strong coupling of the model, and hence we just get a weak type (1,1) estimate on the momenta. See the details in Remark 1.3 and (\ref{3.7})-(\ref{3.8}). On the other hand, the pressure difference across the two-fluid interface endows the low-frequency Green's function with two wave operators, thereby generating two distinct Huygens waves—designated as Huygens wave-I and Huygens wave--I\!I—in the low-frequency regime. Together with the inherent non-conservativity of the model, this bimodal wave structure poses substantial challenges for closing nonlinear estimates and rigorously characterizing wave propagation. To overcome these difficulties, we develop novel nonlinear convolution estimates that precisely capture the interactions among the emerging wave patterns (see the proof of Lemma \ref{l 4.2} for details).

Finally, we address the nonlinear coupling stage by combining the pointwise space-time estimates of the Green's function with the Duhamel principle, alongside the nonlinear estimates established in this paper for the non-conservative model \eqref{1.1}. This allows us to obtain the desired pointwise description of the solution to the full nonlinear problem. The key novelty lies in the effective use of pointwise Green's function bounds to handle the non-conservative coupling, which goes beyond standard energy or Fourier-based methods.

\section{\leftline {\bf{Green's function}}}
\setcounter{equation}{0}
\subsection{Linearization and Reformulation}  Without loss of generality, we set ${\bar R}^\pm=1$ and
\[ n^{\pm}=R^{\pm}-1,
\]
then the Cauchy problem \eqref{1.14}--\eqref{1.15} can be reformulated as
\begin{equation}\label{2.1}
\left\{\begin{array}{l}
\partial_{t} n^++\operatorname{div}\mathbf{m}^+=0, \\
\partial_{t}\mathbf{m}^{+}+\beta_1\nabla n^++\beta_2\nabla
n^--\nu^+_1\Delta \mathbf{m}^+-\nu^+_2\nabla\operatorname{div} \mathbf{m}^+=F_1, \\
\partial_{t} n^-+\operatorname{div}\mathbf{m}^-=0, \\
\partial_{t}\mathbf{m}^{-}+\beta_3\nabla n^++\beta_4\nabla
n^--\nu^-_1\Delta \mathbf{m}^--\nu^-_2\nabla\operatorname{div} \mathbf{m}^-=F_2.
\end{array}\right.
\end{equation}
Note that (\ref{1.23}) together with the stability condition (\ref{1.12}) yields that
\begin{equation}\label{2.1(1)}
\beta_1\beta_4>\beta_2\beta_3.
\end{equation}

The nonlinear terms are given by
\begin{equation}\label{2.2}
\arraycolsep=1.5pt
\begin{array}{rl}
F_{1}=&-\operatorname{div}\big(\frac{\mathbf{m}^+\otimes \mathbf{m}^+}{n^++1}\big)+\sigma^+R^+\nabla\Delta R^+\\
&+\mu^+{\rm div}\Big\{\frac{n^+}{\rho^+}\nabla\frac{\mathbf{m}^+}{n^++1}-\frac{1}{\rho^+}\nabla\frac{n^+\mathbf{m}^+}{n^++1}+\big(\frac{1}{\rho^+}-\frac{1}{\bar{\rho}^+}\big)\nabla \mathbf{m}^+\Big\} \\
&+\mu^+{\rm div}\Big\{\frac{n^+}{\rho^+}\nabla^T\frac{\mathbf{m}^+}{n^++1}-\frac{1}{\rho^+}\nabla^T\frac{n^+\mathbf{m}^+}{n^++1}+\big(\frac{1}{\rho^+}-\frac{1}{\bar{\rho}^+}\big)\nabla^T \mathbf{m}^+\Big\} \\
&+\lambda^+{\rm div}\Big\{\frac{n^+}{\rho^+}{\rm div}\frac{\mathbf{m}^+}{n^++1}I_3-\frac{1}{\rho^+}{\rm div}\frac{n^+\mathbf{m}^+}{n^++1}I_3+\Big(\frac{1}{\rho^+}-\frac{1}{\bar{\rho}^+}\Big){\rm div} \mathbf{m}^+I_3\Big\}\\
&-\underbrace{\bigg\{\mathcal{C}^{2}n^+\Big(\frac{\rho^-}{\rho^+}\nabla n^++\nabla n^-\Big)+\Big(\frac{\mathcal{C}^2\rho^-}{\rho^+}-\frac{\mathcal{C}^2\rho^-}{\rho^+}(1,1)\Big)\nabla n^++(\mathcal{C}^2-\mathcal{C}^2(1,1))\nabla n^-\bigg\}}_{\alpha^+\nabla P-(\beta_1\nabla n^++\beta_2\nabla n^-):=Q_1},
\end{array}
\end{equation}
and
\begin{equation}\label{2.3}
\arraycolsep=1.5pt
\begin{array}{rl}
F_{2}=&-\operatorname{div}\big(\frac{\mathbf{m}^-\otimes \mathbf{m}^-}{n^-+1}\big)+\sigma^-R^-\nabla\Delta R^-\\
&+\mu^-{\rm div}\Big\{\frac{n^-}{\rho^-}\nabla\frac{\mathbf{m}^-}{n^-+1}-\frac{1}{\rho^-}\nabla\frac{n^-\mathbf{m}^-}{n^-+1}+\big(\frac{1}{\rho^-}-\frac{1}{\bar{\rho}^-}\big)\nabla \mathbf{m}^-\Big\} \\
&+\mu^-{\rm div}\Big\{\frac{n^-}{\rho^-}\nabla^T\frac{\mathbf{m}^-}{n^-+1}-\frac{1}{\rho^-}\nabla^T\frac{n^-\mathbf{m}^-}{n^-+1}+\big(\frac{1}{\rho^-}-\frac{1}{\bar{\rho}^-}\big)\nabla^T \mathbf{m}^-\Big\} \\
&+\lambda^-{\rm div}\Big\{\frac{n^-}{\rho^-}{\rm div}\frac{\mathbf{m}^-}{n^-+1}I_3-\frac{1}{\rho^-}{\rm div}\frac{n^-\mathbf{m}^-}{n^-+1}I_3+\big(\frac{1}{\rho^-}-\frac{1}{\bar{\rho}^-}\big){\rm div} \mathbf{m}^-I_3\Big\}\\
&-\underbrace{\bigg\{\mathcal{C}^{2}n^-\big(\nabla n^++\frac{\rho^-}{\rho^+}\nabla n^-\big)+\big(\frac{\mathcal{C}^2\rho^+}{\rho^-}-\frac{\mathcal{C}^2\rho^+}{\rho^-}(1,1)\big)\nabla n^-+(\mathcal{C}^2-\mathcal{C}^2(1,1))\nabla n^+\bigg\}}_{\alpha^-\nabla P-(\beta_3\nabla n^++\beta_4\nabla n^-):=Q_2},
\end{array}
\end{equation}
where the last terms $Q_1$ and $Q_2$ in $F_1$ and $F_2$ have been given in (\ref{1.29}) in another simplified  form.

In terms of the semigroup theory for evolutionary equation, we 
investigate the following initial value problem for the
corresponding linear system of \eqref{2.1} on the unknown $U:=(n^+,\mathbf{m}^+,n^-,\mathbf{m}^-)$:
\begin{equation}
\begin{cases}
U_t=\mathcal BU,\\
U\big|_{t=0}={U}_0,
\end{cases}   \label{2.4}
\end{equation}
where the operator $\mathcal B$ is given by
\begin{equation}\nonumber\mathcal B=\begin{pmatrix}
0&-\text{div}&0&0\\
-\beta_1\nabla &\nu_{1}^{+} \Delta+\nu_{2}^{+} \nabla \otimes \nabla&-\beta_2\nabla&0\\
0&0&0&-\text{div}\\
-\beta_3\nabla&0&-\beta_4\nabla &\nu_{1}^{-}
\Delta+\nu_{2}^{-} \nabla \otimes \nabla
\end{pmatrix}.
\end{equation}
Applying Fourier transform to the system \eqref{2.4}, one has
\begin{equation}\label{2.5}
\begin{cases}
\widehat {{U}}_t=\mathcal B(\xi)\widehat {U},\\
\widehat {U}\big|_{t=0}=\widehat U_0=(\widehat {n^+_0}, \widehat{ \mathbf{m}^+_0},
\widehat {n^-_0}, \widehat {\mathbf{m}^-_0}),
\end{cases}
\end{equation}
where $\widehat
{U}(\xi,t)=\mathcal{F}(U(x,t))$,
$\xi=(\xi_1,\xi_2,\xi_3)^T$ and $\mathcal B(\xi)$ is defined by
\begin{equation}\label{2.6}
\mathcal B(\xi)=\begin{pmatrix}
0&-{i}\xi^T&0&0\\
-{i}\beta_1\xi&-\nu_1^+|\xi|^2{\mathbb I}_{3\times 3}-\nu_2^+\xi\otimes\xi&-{i}\beta_2\xi&0\\
0&0&0&- {i}\xi^T\\
- {i}\beta_3\xi&0&-
{i}\beta_4 &-\nu_1^-|\xi|^2{\mathbb I}_{3\times
3}-\nu_2^-\xi\otimes\xi
\end{pmatrix}.
\end{equation}
To facilitate narrative in later use, we also use the definition of Green's function $G(x,t)$  with the following standard form as our previous works:
\begin{equation}\label{2.6(1)}
\left\{\begin{array}{l}
G_{t}=\mathcal{A} G, \\
G|_{t=0}=\delta_0(x)\mathbb{I}_{8\times8}.
\end{array}\right.
\end{equation}

Since the system \eqref{2.4} has eight equations, we
take Hodge decomposition to system \eqref{2.4} such that it can be
decoupled into three systems. One has four equations, and the other
two are classic heat equations. This key observation allows us to
derive the optimal linear convergence rates.

To begin with, let $\varphi^{\pm}=\Lambda^{-1}{\rm
div}\mathbf{m}^{\pm}$ be the ``compressible part" of the momenta
$\mathbf{m}^{\pm}$, and denote $\phi^{\pm}=\Lambda^{-1}{\rm
curl}\mathbf{m}^{\pm}$ (with $({\rm curl} z)_i^j
=\partial_{x_j}z^i-\partial_{x_i}z^j$) by the ``incompressible part"
of the momenta $\mathbf{m}^{\pm}$. Then, we can divide the system
\eqref{2.4} into the compressible part
\begin{equation}\label{2.7}
\begin{cases}
\partial_t{n^+}+\Lambda{\varphi^+}=0,\\
\partial_t{\varphi^+}-\beta_1\Lambda{n^+}-\beta_2\Lambda{n^-}+\nu^+\Lambda^2{\varphi^+}=0,\\
\partial_t{n^-}+\Lambda{\varphi^-}=0,\\
\partial_t{\varphi^-}-\beta_3\Lambda{n^+}-\beta_4\Lambda{n^-}+\nu^-\Lambda^2{\varphi^-}=0,\\
(n^+, \varphi^+, n^-, \varphi^-)\big|_{t=0}=({n}^+_0, \Lambda^{-1}{\rm div}{\mathbf{m}}^{+}_0, {n}^-_0, \Lambda^{-1}{\rm div}{\mathbf{m}}^{-}_0)(x),\\
\end{cases}
\end{equation}
with its Green's function $\tilde{G}(x,t)$,
and the incompressible part
\begin{equation}\label{2.8}
\begin{cases}
\partial_t\phi^++\nu^+_1\Lambda^2\phi^+=0,\\
\partial_t\phi^-+\nu^-_1\Lambda^2\phi^-=0,\\
(\phi^+,\phi^-)\big|_{t=0}=(\Lambda^{-1}{\rm curl}{\mathbf{m}}^{+}_0,
\Lambda^{-1}{\rm curl}{\mathbf{m}}^{-}_0)(x)
\end{cases}
\end{equation}
where $\nu^{\pm}=\nu^{\pm}_1+\nu^{\pm}_2$. It is obvious that the incompressible part are two heat kernels:
\begin{eqnarray}\label{2.8(1)}
\hat{\phi}^+=e^{-\nu_1^+|\xi|^2t},\ \hat{\phi}^-=e^{-\nu_1^-|\xi|^2t},
\end{eqnarray}
and hence, we mainly focus on the compressible part of Green's function.

\subsection{Compressible part}

 In view of the semigroup
theory, we may represent the IVP \eqref{2.4} for $\mathcal
U=(n^+, \varphi^+, n^-, \varphi^-)^T$ as
\begin{equation}\label{2.9}
\begin{cases}
\mathcal U_t=\mathcal A\mathcal U,\\
\mathcal U\big|_{t=0}=\mathcal U_0,
\end{cases}
\end{equation}
where the operator $\mathcal A$ is defined by
\begin{equation}\nonumber\mathcal A=\begin{pmatrix}
0&-\Lambda&0&0\\
\beta_1\Lambda&-\nu^+\Lambda^2&\beta_2\Lambda&0\\
0&0&0&-\Lambda\\
\beta_3\Lambda&0&\beta_4\Lambda&-\nu^-\Lambda^2
\end{pmatrix}.
\end{equation}
Taking Fourier transform to system \eqref{2.9}, we obtain
\begin{equation}\label{2.10}
\begin{cases}
\widehat {\mathcal U}_t=\mathcal A(\xi)\widehat {\mathcal U},\\
\widehat {\mathcal U}\ \big|_{t=0}=\widehat {\mathcal U}_0,
\end{cases}
\end{equation}
where $\widehat {\mathcal U}(\xi,t)$ is the Fourier transform of ${\mathcal U}(x,t)$
and $\mathcal A(\xi)$ is given by
\begin{equation}\label{2.11}
\mathcal A(\xi)=\begin{pmatrix}
0&-|\xi|&0&0\\
\beta_1|\xi|&-\nu^+|\xi|^2&\beta_2|\xi|&0\\
0&0&0&-|\xi|\\
\beta_3|\xi|&0&\beta_4|\xi|&-\nu^-|\xi|^2
\end{pmatrix}.
\end{equation}
Its eigenvalues satisfy
\begin{equation}\label{2.12}
\begin{array}{rl}
{\rm det}(\lambda{\rm I}-\mathcal A(\xi))
=&\lambda^4+(\nu^+|\xi|^2+\nu^-|\xi|^2)\lambda^3+(\beta_1|\xi|^2+\beta_4|\xi|^2+\nu^+\nu^-|\xi|^4)\lambda^2\\
&\quad+(\beta_1\nu^-|\xi|^4+\beta_4\nu^+|\xi|^4)\lambda+(\beta_1\beta_4-\beta_2\beta_3)|\xi|^4\\
=&0,
\end{array}
\end{equation}
which implies that the matrix $\mathcal A(\xi)$ possesses four different
eigenvalues:
\begin{equation}\nonumber
 \lambda_1=\lambda_1(|\xi|),\quad \lambda_2=\lambda_2(|\xi|),\quad
 \lambda_3=\lambda_3(|\xi|),\quad \lambda_4=\lambda_4(|\xi|).
\end{equation}
Consequently, the semigroup $e^{t\mathcal A}$ can be decomposed
into
\begin{equation}\label{2.21}
 \text{e}^{t\mathcal A(\xi)}=\sum_{i=1}^4\text{e}^{\lambda_it}P^i(\xi),
\end{equation}
where the projector $P^i(\xi)$ is defined by
\begin{equation}\label{2.22}
 P^i(\xi)=\prod_{j\neq i}\frac{\mathcal A(\xi)-\lambda_jI}{\lambda_i-\lambda_j}, \quad i,j=1,2,3,4.
\end{equation}
Thus, the solution of IVP \eqref{2.10} can be expressed as
\begin{equation}\label{2.23}
\widehat {\mathcal U}(\xi,t)=\text{e}^{t\mathcal A(\xi)}\widehat {\mathcal U}_0(\xi)=\left(\sum_{i=1}^4
\text{e}^{\lambda_it}P^i(\xi)\right)\widehat {\mathcal U}_0(\xi).
\end{equation}

\vspace{3.8mm}
\textbf{\textit{Low frequency part.}}\vspace{2mm}

From discriminant of the quartic equation in one unknown, we know that when $|\xi|\ll1$ there exist two different real roots and a pair of conjugate imaginary roots. Then, by using the method of undetermined coefficients, we have the following for the spectral in the low frequency part:
\begin{lemma}\label{lemma2.1}
There exists a positive constants $\eta_1\ll 1 $ such that, for
$|\xi|\leq \eta_1$, the spectral has the following Taylor series
expansion:
\begin{equation}\label{2.24}
\left\{\begin{array}{lll}\displaystyle \lambda_{1}=\overline{\lambda_2}=-\left[\frac{\nu^++\nu^-}{4}-\frac{(\nu^+-\nu^-)(\beta_1-\beta_4)}{8\kappa_1}\right]|\xi|^2+{\rm i}\sqrt{\kappa_2-\kappa_1}|\xi|+\mathcal O(|\xi|^3),\\[3mm]
\displaystyle  \lambda_3=\overline{\lambda_4}=-\left[\frac{\nu^++\nu^-}{4}+\frac{(\nu^+-\nu^-)(\beta_1-\beta_4)}{8\kappa_1}\right]|\xi|^2
+{\rm i}\sqrt{\kappa_2+\kappa_1}|\xi|+\mathcal O(|\xi|^3),
\end{array}\right.
\end{equation}
where $\kappa_1=\sqrt{\frac{(\beta_1+\beta_4)^2}{4}-(\beta_1\beta_4-\beta_2\beta_3)}$, $\displaystyle\kappa_2=\frac{\beta_1+\beta_4}{2}$ and $\beta_i\ (i=1,2,3,4)$ are defined in (\ref{2.1(1)}).
\end{lemma}
Note that all of the real parts of $\lambda_i$ with $i=1,2,3,4$ in Lemma \ref{lemma2.1} are positive due to the fact $\beta_1\beta_4-\beta_2\beta_3>0$ according to the condition (\ref{1.12}). Additionally, we denote the propagation speeds ${\rm c}_1\triangleq\sqrt{\kappa_2-\kappa_1}<\sqrt{\kappa_2+\kappa_1}\triangleq{\rm c}_2$ for later use.

By virtue of formula \eqref{2.23} and Taylor series expansion of $\lambda_i$ $(1\leq i\leq 4)$ in \eqref{2.24}, we can represent the low frequency of $P^i$ as follows:
{\small \begin{equation}\label{2.26}\begin{split}P^{1,l}(\xi)=&\begin{pmatrix}
\frac{\kappa_1+\kappa_2-\beta_1}{4\kappa_1}&
\frac{\kappa_1+\kappa_2-\beta_1}{4\kappa_1\sqrt{\kappa_2-\kappa_1}}{\rm{i}}&-\frac{\beta_2}{4\kappa_1}
&\frac{\beta_2}{4\kappa_1\sqrt{\kappa_2-\kappa_1}}{\rm{i}}\\
\frac{\beta_1^2+\beta_2\beta_3-\beta_1(\kappa_1+\kappa_2)}{4\kappa_1
\sqrt{\kappa_2-\kappa_1}}{\rm{i}} & \frac{\kappa_1+\kappa_2-\beta_1}{4\kappa_1} & \frac{\beta_1\beta_2+\beta_2\beta_4-\beta_2(\kappa_1+\kappa_2)}{4\kappa_1\sqrt{\kappa_2-\kappa_1}}{\rm{i}}&\frac{-\beta_2}{4\kappa_1}\\
\frac{-\beta_3}{4\kappa_1}& -\frac{\beta_3}{4\kappa_1\sqrt{\kappa_2-\kappa_1}}{\rm{i}}& \frac{\kappa_1+\kappa_2-\beta_4}{4\kappa_1} & \frac{\kappa_1+\kappa_2-\beta_4}{4\kappa_1\sqrt{\kappa_2-\kappa_1}}{\rm{i}}\\
\frac{\beta_1\beta_3+\beta_3\beta_4-\beta_3(\kappa_1+\kappa_2)}{4\kappa_1\sqrt{\kappa_2-\kappa_1}}{\rm{i}} &-\frac{\beta_3}{4\kappa_1} & \frac{\beta_2\beta_3+\beta_4^2-\beta_4(\kappa_1+\kappa_2)}{4\kappa_1\sqrt{\kappa_2-\kappa_1}}{\rm{i}}
&\frac{\kappa_1+\kappa_2-\beta_4}{4\kappa_1}
\end{pmatrix}+\mathcal O(|\xi|),\end{split}\end{equation}

\begin{equation}\label{2.27}\begin{split}P^{2,l}(\xi)=&\begin{pmatrix}
\frac{\kappa_1+\kappa_2-\beta_1}{4\kappa_1} \!&\!
-\frac{\kappa_1+\kappa_2-\beta_1}{4\kappa_1\sqrt{\kappa_2-\kappa_1}}{\rm{i}} \!&\! -\frac{\beta_2}{4\kappa_1}
\!&\!-\frac{\beta_2}{4\kappa_1\sqrt{\kappa_2-\kappa_1}}{\rm{i}}\\
-\frac{\beta_1^2+\beta_2\beta_3-\beta_1(\kappa_1+\kappa_2)}{4\kappa_1
\sqrt{\kappa_2-\kappa_1}}{\rm{i}} \!&\! \frac{\kappa_1+\kappa_2-\beta_1}{4\kappa_1}\! &\! -\frac{\beta_1\beta_2+\beta_2\beta_4-\beta_2(\kappa_1+\kappa_2)}{4\kappa_1\sqrt{\kappa_2-\kappa_1}}{\rm{i}} \!&\! \frac{-\beta_2}{4\kappa_1}\\
\frac{-\beta_3}{4\kappa_1}\!&\! -\frac{\beta_3}{4\kappa_1\sqrt{\kappa_2-\kappa_1}}{\rm{i}}\!&\! \frac{\kappa_1+\kappa_2-\beta_4}{4\kappa_1} \!&\! -\frac{\kappa_1+\kappa_2-\beta_4}{4\kappa_1\sqrt{\kappa_2-\kappa_1}}{\rm{i}}\\
-\frac{\beta_1\beta_3+\beta_3\beta_4-\beta_3(\kappa_1+\kappa_2)}{4\kappa_1\sqrt{\kappa_2-\kappa_1}}{\rm{i}} &-\frac{\beta_3}{4\kappa_1} \!&\! -\frac{\beta_2\beta_3+\beta_4^2-\beta_4(\kappa_1+\kappa_2)}{4\kappa_1\sqrt{\kappa_2-\kappa_1}}{\rm{i}}
\!&\!\frac{\kappa_1+\kappa_2-\beta_4}{4\kappa_1}
\end{pmatrix}+\mathcal O(|\xi|),\end{split}\end{equation}

\begin{equation}\label{2.28}\begin{split}P^{3,l}(\xi)=&\begin{pmatrix}
\frac{\beta_1-(\kappa_2-\kappa_1)}{4\kappa_1}&
\frac{\beta_1-(\kappa_2-\kappa_1)}{4\kappa_1\sqrt{\kappa_2+\kappa_1}}{\rm{i}}&\frac{\beta_2}{4\kappa_1}
&\frac{\beta_2}{4\kappa_1\sqrt{\kappa_2+\kappa_1}}{\rm{i}}\\
-\frac{\beta_1^2+\beta_2\beta_3-\beta_1(\kappa_2-\kappa_1)}{4\kappa_1
\sqrt{\kappa_2+\kappa_1}}{\rm{i}} & \frac{\beta_1-(\kappa_2-\kappa_1)}{4\kappa_1} & -\frac{\beta_1\beta_2+\beta_2\beta_4-\beta_2(\kappa_2-\kappa_1)}{4\kappa_1\sqrt{\kappa_2+\kappa_1}}{\rm{i}}&
\frac{\beta_2}{4\kappa_1}\\
\frac{\beta_3}{4\kappa_1}& \frac{\beta_3}{4\kappa_1\sqrt{\kappa_2+\kappa_1}}{\rm{i}}& \frac{\beta_4-(\kappa_2-\kappa_1)}{4\kappa_1} & \frac{\beta_4-(\kappa_2-\kappa_1)}{4\kappa_1\sqrt{\kappa_2+\kappa_1}}{\rm{i}}\\
-\frac{\beta_1\beta_3+\beta_3\beta_4-\beta_3(\kappa_2-\kappa_1)}{4\kappa_1\sqrt{\kappa_2+\kappa_1}}{\rm{i}} &\frac{\beta_3}{4\kappa_1} & -\frac{\beta_2\beta_3+\beta_4^2-\beta_4(\kappa_2-\kappa_1)}{4\kappa_1\sqrt{\kappa_2+\kappa_1}}{\rm{i}}
&\frac{\beta_4-(\kappa_2-\kappa_1)}{4\kappa_1}
\end{pmatrix}+\mathcal O(|\xi|),\end{split}\end{equation}

and
\begin{equation}\label{2.29}\begin{split}P^{4,l}(\xi)=&\begin{pmatrix}
\frac{\beta_1-(\kappa_2-\kappa_1)}{4\kappa_1}&
-\frac{\beta_1-(\kappa_2-\kappa_1)}{4\kappa_1\sqrt{\kappa_2+\kappa_1}}{\rm{i}}&\frac{\beta_2}{4\kappa_1}
&-\frac{\beta_2}{4\kappa_1\sqrt{\kappa_2+\kappa_1}}{\rm{i}}\\
\frac{\beta_1^2+\beta_2\beta_3-\beta_1(\kappa_2-\kappa_1)}{4\kappa_1
\sqrt{\kappa_2+\kappa_1}}{\rm{i}} & \frac{\beta_1-(\kappa_2-\kappa_1)}{4\kappa_1} & \frac{\beta_1\beta_2+\beta_2\beta_4-\beta_2(\kappa_2-\kappa_1)}{4\kappa_1\sqrt{\kappa_2+\kappa_1}}{\rm{i}}&
\frac{\beta_2}{4\kappa_1}\\
\frac{\beta_3}{4\kappa_1}& -\frac{\beta_3}{4\kappa_1\sqrt{\kappa_2+\kappa_1}}{\rm{i}}& \frac{\beta_4-(\kappa_2-\kappa_1)}{4\kappa_1} & -\frac{\beta_4-(\kappa_2-\kappa_1)}{4\kappa_1\sqrt{\kappa_2+\kappa_1}}{\rm{i}}\\
\frac{\beta_1\beta_3+\beta_3\beta_4-\beta_3(\kappa_2-\kappa_1)}{4\kappa_1\sqrt{\kappa_2+\kappa_1}}{\rm{i}} &\frac{\beta_3}{4\kappa_1} & \frac{\beta_2\beta_3+\beta_4^2-\beta_4(\kappa_2-\kappa_1)}{4\kappa_1\sqrt{\kappa_2+\kappa_1}}{\rm{i}}
&\frac{\beta_4-(\kappa_2-\kappa_1)}{4\kappa_1}
\end{pmatrix}+\mathcal O(|\xi|).\end{split}\end{equation}


\vspace{2.5mm}
{\bf High frequency part.}

 From the  characteristic equation (\ref{2.12}), we have the following expansion of spectra in the high frequency.
\begin{lemma}\label{l 2.2} There exists a positive constant  $K \gg 1$  such that, for  $|\xi| \geq K$, the spectrum has the following Taylor series expansion:
\begin{equation}\label{2.35}
\begin{cases}\lambda_{1}=-\frac{(\beta_1\nu^-+\beta_4\nu^+)+\sqrt{(\beta_1\nu^-+\beta_4\nu^+)^2-4\nu^+\nu^-(\beta_1\beta_4-\beta_2\beta_3)}}{2\nu^+\nu^-}+\mathcal O(1),\\[2mm]
\lambda_{2}=-\nu^+|\xi|^2+\mathcal O(1),\\[2mm]
\lambda_{3}=-\frac{(\beta_1\nu^-+\beta_4\nu^+)-\sqrt{(\beta_1\nu^-+\beta_4\nu^+)^2-4\nu^+\nu^-(\beta_1\beta_4-\beta_2\beta_3)}}{2\nu^+\nu^-}+\mathcal O(1),\\[2mm]
\lambda_{4}=-\nu^-|\xi|^2+\mathcal O(1).
\end{cases}
\end{equation}
\end{lemma}

After a direct and tedious computation, we can get the following expansion of the projection $P^i$ with $i=1,2,3,4$ in the high frequency according to Lemma \ref{l 2.2}:
\begin{align}\label{2.36}
P^{i}(\xi)=&\begin{pmatrix}
\mathcal{O}(1)&\mathcal{O}(|\xi|^{-1})&\mathcal{O}(1)&\mathcal{O}(|\xi|^{-1})\\
\mathcal{O}(|\xi|^{-1})&\mathcal{O}(1)&\mathcal{O}(|\xi|^{-1})&\mathcal{O}(1)\\
\mathcal{O}(1)&\mathcal{O}(|\xi|^{-1})&\mathcal{O}(1)&\mathcal{O}(|\xi|^{-1})\\
\mathcal{O}(|\xi|^{-1})&\mathcal{O}(1)&\mathcal{O}(|\xi|^{-1})&\mathcal{O}(1)
\end{pmatrix}+\cdots.
\end{align}
One can see that the high frequency of this two-fluid model is consistent with the classical compressible insentropic Navier-Stokes equations in \cite{Ls,Lw}, that is, the residual part of the high-frequency component of Green's function, with its singularity removed, satisfies the following spacetime estimate, while the singularity of the high-frequency component can be characterized by the delta function or the heat kernel. In fact, from (\ref{2.35})-(\ref{2.36}), we find that there exist the terms like $\mathcal{O}(1)e^{-t/C}$, $\mathcal{O}(|\xi|^{-1})e^{-t/C}$ and $\mathcal{O}(1)e^{-\frac{|\xi|^2t}{C}}$ in the high frequency of Green's function, then according to Lemma \ref{A.2}, we have
\begin{lemma}\label{l 2.2(1)}
For any integer $N>0$, the high frequency $\tilde{G}^h(x,t)$ also satisfies
\begin{equation}\label{2.37}
|\tilde{G}^h(x,t)-G_{S}|\lesssim e^{-t/C}(1+|x|^2)^{-N}.
\end{equation}
Here the singular term $G_S$ contains two parts: $G_{S1}=e^{-t/C}(C_0\delta(x)+f(x))$ with $\|f\|_{L^1}\leq C,\ {\rm supp}f(x)\subset\{x;|x|<\eta_0\ll1\}$, and $G_{S2}=e^{-t/C}t^{-\frac{3}{2}}e^{-\frac{|x|^2}{t}}$.
\end{lemma}
\begin{remark}\label{r 2.1} The incompressible part of the system (\ref{1.1}) obeys two decoupled heat equations,  which together with the compressible part in Lemma \ref{l 2.2(1)} yields that the high frequency of Green's function $G(x,t)$ of the full system (\ref{1.1}) has the same estimates as in Lemma \ref{l 2.2(1)}. Besides, since $G_{S2}$ is like the heat kernel with exponential decay rate, it is evident that $G_{S2}$ can be treated the same as  $G_{S1}$ when handling the convolution of the singular part with the initial value or the nonlinear term. This will not be repeated hereafter.

\end{remark}

\vspace{2.5mm}
{\bf Middle frequency part.}

\vspace{2.5mm}
For the middle frequency part, we only need to prove the following two lemmas. The first one is about the positivity of the real part of spectra based on the Routh-Hurwitz theorem.
\begin{lemma}\label{l 2.3}
When $\eta_1\leq|\xi|\leq K$ with two fixed positive constants $\eta_1$ and $K$, there exists a positive constant $b$ such that
\begin{equation}\label{2.41}
\begin{array}{ll}
{\rm Re}(\lambda_1(|\xi|),\lambda_2(|\xi|),\lambda_3(|\xi|),\lambda_4(|\xi|))\leq -b.
\end{array}
\end{equation}
\end{lemma}

The analyticity for the middle frequency is partially based on the idea in Li \cite{ld}.
\begin{lemma}\label{l 2.4}
The middle part $\hat{\tilde{G}}^m(\xi,t)$ is analytic when $|\xi|^2\geq \delta$, where $\delta$ is a positive constant.
\end{lemma}

Lemma \ref{l 2.3} and Lemma \ref{l 2.4} directly yield the space-time estimates for middle frequency $\tilde{G}^m$:
\begin{lemma}\label{l 2.5}For any $|\alpha|\geq0$, there exists a constant $b>0$ such that
\begin{equation*}
|D_x^\alpha \tilde{G}^m(x,t)|\leq Ce^{-bt}(1+|x|^2)^{-N},
\end{equation*}
where the constant $N>0$ can be arbitrarily large.
\end{lemma}

%

\section{Pointwise space-time description of Green's function}
\quad\quad We shall bring Green's function in Fourier space back to the physical space, and combine the compressible part and the incompressible part to obtain the pointwise description of Green's function of the original system (\ref{2.2}). The difficulties mainly include resolving the singularity from the Riesz operator in low frequency of Green's matrix by using suitable combinations, and giving the description of the singular part in high frequency arising from the definition of Green's function.

The Hodge decomposition, Lemma \ref{lemma2.1} together with (\ref{2.8(1)}) and (\ref{2.26})-(\ref{2.29}) directly give the following for the low frequency:
\begin{align}\label{3.1}
\hat{n}^{+,l}=&\Big[\frac{(\kappa_1+\kappa_2-\beta_1)(e^{\lambda_1t}+e^{\lambda_2t})}{4\kappa_1}+\frac{(\beta_1-(\kappa_2-\kappa_1))(e^{\lambda_3t}+e^{\lambda_4t})}{4\kappa_1}
\Big]\hat{n}_0^+\nonumber\\
\!\!&\!+\!\Big[\rm{i}\frac{(\kappa_1+\kappa_2-\beta_1)(e^{\lambda_1t}\!-\!e^{\lambda_2t})}{4\kappa_1\sqrt{\kappa_2-\kappa_1}}
+\frac{(\beta_1-(\kappa_2-\kappa_1))(e^{\lambda_3t}-e^{\lambda_4t})}{4\kappa_1\sqrt{\kappa_2+\kappa_1}}\Big)\Big]\frac{{\rm i}\xi^T}{|\xi|}\hat{\mathbf{m}}_{0}^{+}\nonumber\\[2mm]
&+\Big[\frac{-\beta_2(e^{\lambda_1t}\!+\!e^{\lambda_2t})}{4\kappa_1}+\frac{\beta_2(e^{\lambda_3t}\!+\!e^{\lambda_4t})}{4\kappa_1}\Big]\hat{n}_0^{-} \nonumber\\[2mm]
&+\Big[\frac{\beta_2(e^{\lambda_1t}-e^{\lambda_2t})}{4\kappa_1\sqrt{\kappa_2-\kappa_1}}{\rm{i}}
+\frac{\beta_2(e^{\lambda_3t}-e^{\lambda_4t})}{4\kappa_1\sqrt{\kappa_2+\kappa_1}}{\rm{i}}\Big]\frac{\rm{i}\xi^T}{|\xi|}\hat{\mathbf{m}}_{0}^{-}+\mathcal{R}_1^l\nonumber\\[3mm]
:=&(G_{11}^l,G_{12}^l,G_{13}^l,G_{14}^l)\cdot(\hat{n}_{0}^+,\hat{\mathbf{m}}_0^+,\hat{n}_{0}^-,\hat{\mathbf{m}}_{0}^-),
\end{align}
\begin{align}\label{3.2}
\!\!\!\!\!\!\!\!\!\!\!\!\!\!\!\!\!\!\!\!\!\!\!\!\!\!\!&\!\!\!\!\!\!\hat{\mathbf{m}}^{+,l} =-\widehat{\Lambda^{-1}\nabla \varphi^+}-\widehat{\Lambda ^{-1}{\rm div} \phi^+}\nonumber\\
=&\Big[\frac{\beta_1^2+\beta_2\beta_3-\beta_1(\kappa_1+\kappa_2)}{4\kappa_1
\sqrt{\kappa_2-\kappa_1}}(e^{\lambda_1t}-e^{\lambda_2t})-\frac{\beta_1^2+\beta_2\beta_3-\beta_1(\kappa_2-\kappa_1)}{4\kappa_1
\sqrt{\kappa_2+\kappa_1}}(e^{\lambda_3t}-e^{\lambda_4t})\Big]\frac{\rm{i}\xi\hat{n}_{0}^{+}}{|\xi|}\nonumber\\
&+\bigg[\frac{(\kappa_1+\kappa_2-\beta_1)(e^{\lambda_1t}\!+\!e^{\lambda_2t})}{4\kappa_1}\frac{\xi\xi^T}{|\xi|^{2}}+\frac{(\beta_1-(\kappa_2-\kappa_1))(e^{\lambda_3t}\!+\!e^{\lambda_4t})}{4\kappa_1}
\frac{\xi\xi^T}{|\xi|^{2}}+e^{-\nu^+|\xi|^2t}\Big(\mathbb{I}-\frac{\xi\xi^T}{|\xi|^{2}}\Big)\bigg]\hat{\mathbf{m}}_{0}^+\nonumber\\
&+\Big[\frac{\beta_1\beta_2+\beta_2\beta_4-\beta_2(\kappa_1+\kappa_2)}{4\kappa_1\sqrt{\kappa_2-\kappa_1}}(e^{\lambda_1t}-e^{\lambda_2t})
-\frac{\beta_1\beta_2+\beta_2\beta_4-\beta_2(\kappa_2-\kappa_1)}{4\kappa_1\sqrt{\kappa_2+\kappa_1}}(e^{\lambda_3t}-e^{\lambda_4t})
\Big]\frac{\rm{i}\xi\hat{n}_{0}^{-}}{|\xi|}\nonumber\\[-0.5mm]
&\!+\!\bigg[\frac{-\beta_2(e^{\lambda_1t}\!+\!e^{\lambda_2t})}{4\kappa_1}\frac{\xi\xi^T}{|\xi|^{2}}+\frac{\beta_2(e^{\lambda_3t}\!+\!e^{\lambda_4t})}{4\kappa_1}
\frac{\xi\xi^T}{|\xi|^{2}}+e^{-\nu^-|\xi|^2t}\Big(\mathbb{I}-\frac{\xi\xi^T}{|\xi|^{2}}\Big)\bigg]\hat{\mathbf{m}}_{0}^{-}\!+\!\mathcal{R}_2^l\nonumber\\
:=&(G_{21}^l,G_{22}^l,G_{23}^l,G_{24}^l)\cdot(\hat{n}_{0}^+,\hat{\mathbf{m}}_0^+,\hat{n}_{0}^-,\hat{\mathbf{m}}_{0}^-),
\end{align}
\vspace{-4mm}
\begin{align}\label{3.3}
\hat{n}^{-,l}=&\Big[\frac{-\beta_3(e^{\lambda_1t}+e^{\lambda_2t})}{4\kappa_1}+\frac{\beta_3(e^{\lambda_3t}+e^{\lambda_4t})}{4\kappa_1}
\Big]\hat{n}_0^+\nonumber\\
\!\!&\!+\!\Big[\rm{i}\frac{-\beta_3(e^{\lambda_1t}\!-\!e^{\lambda_2t})}{4\kappa_1\sqrt{\kappa_2-\kappa_1}}
+\frac{\beta_3(e^{\lambda_3t}-e^{\lambda_4t})}{4\kappa_1\sqrt{\kappa_2+\kappa_1}}\Big]\frac{{\rm i}\xi^T}{|\xi|}\hat{\mathbf{m}}_{0}^{+}\nonumber\\
&+\Big[\frac{(\kappa_1+\kappa_2-\beta_4)(e^{\lambda_1t}\!+\!e^{\lambda_2t})}{4\kappa_1}-\frac{(\kappa_1+\kappa_2-\beta_4)(e^{\lambda_3t}\!+\!e^{\lambda_4t})}{4\kappa_1}\Big]\hat{n}_0^{-} \nonumber\\[-0.5mm]
&+\Big[\frac{(\kappa_1+\kappa_2-\beta_4)(e^{\lambda_1t}-e^{\lambda_2t})}{4\kappa_1\sqrt{\kappa_2-\kappa_1}}{\rm{i}}
+\frac{(\kappa_1+\kappa_2-\beta_4)(e^{\lambda_3t}-e^{\lambda_4t})}{4\kappa_1\sqrt{\kappa_2+\kappa_1}}{\rm{i}}\Big]\frac{i\xi^T}{|\xi|}\hat{\mathbf{m}}_{0}^{-}+\mathcal{R}_3^l\nonumber\\[-0.5mm]
:=&(G_{31}^l,G_{32}^l,G_{33}^l,G_{34}^l)\cdot(\hat{n}_{0}^+,\hat{\mathbf{m}}_0^+,\hat{n}_{0}^-,\hat{\mathbf{m}}_{0}^-),
\end{align}
\begin{align}\label{3.4}
&\!\!\!\!\!\!\!\!\!\hat{\mathbf{m}}^{-,l}=-\widehat{\Lambda^{-1}\nabla \varphi^-}-\widehat{\Lambda ^{-1}{\rm div} \phi^-}\nonumber\\
=&\Big[\frac{\beta_1\beta_3+\beta_3\beta_4-\beta_3(\kappa_1+\kappa_2)}{4\kappa_1
\sqrt{\kappa_2-\kappa_1}}(e^{\lambda_1t}-e^{\lambda_2t})-\frac{\beta_1\beta_3+\beta_3\beta_4-\beta_3(\kappa_2-\kappa_1)}{4\kappa_1
\sqrt{\kappa_2+\kappa_1}}(e^{\lambda_3t}-e^{\lambda_4t})\Big]\frac{\rm{i}\xi\hat{n}_{0}^{+}}{|\xi|}\nonumber\\
&+\bigg[\frac{-\beta_3(e^{\lambda_1t}\!+\!e^{\lambda_2t})}{4\kappa_1}\frac{\xi\xi^T}{|\xi|^{2}}+\frac{\beta_3(e^{\lambda_3t}\!+\!e^{\lambda_4t})}{4\kappa_1}
\frac{\xi\xi^T}{|\xi|^{2}}+e^{-\nu^+|\xi|^2t}\Big(\mathbb{I}-\frac{\xi\xi^T}{|\xi|^{2}}\Big)\bigg]\hat{\mathbf{m}}_{0}^+\nonumber\\
&+\Big[\frac{\beta_2\beta_3+\beta_4^2-\beta_4(\kappa_1+\kappa_2)}{4\kappa_1\sqrt{\kappa_2-\kappa_1}}(e^{\lambda_1t}-e^{\lambda_2t})
-\frac{\beta_2\beta_3+\beta_4^2-\beta_4(\kappa_2-\kappa_1)}{4\kappa_1\sqrt{\kappa_2+\kappa_1}}(e^{\lambda_3t}-e^{\lambda_4t})
\Big]\frac{\rm{i}\xi\hat{n}_{0}^{-}}{|\xi|}\nonumber\\[-0.5mm]
&\!+\!\bigg[\frac{(\kappa_1+\kappa_2-\beta_4)(e^{\lambda_1t}\!+\!e^{\lambda_2t})}{4\kappa_1}\frac{\xi\xi^T}{|\xi|^{2}}+\frac{(\beta_4-(\kappa_2-\kappa_1))(e^{\lambda_3t}\!+\!e^{\lambda_4t})}{4\kappa_1}
\frac{\xi\xi^T}{|\xi|^{2}}+e^{-\nu^-|\xi|^2t}\Big(\mathbb{I}-\frac{\xi\xi^T}{|\xi|^{2}}\Big)\bigg]\hat{\mathbf{m}}_{0}^{-}\!+\!\mathcal{R}_4^l\nonumber\\
:=&(G_{21}^l,G_{22}^l,G_{23}^l,G_{24}^l)\cdot(\hat{n}_{0}^+,\hat{\mathbf{m}}_0^+,\hat{n}_{0}^-,\hat{\mathbf{m}}_{0}^-).
\end{align}
We can mainly deal with the leading term of each entry, since the rest terms $\mathcal{R}_i^l$ for $i=1,2,3,4$  in (\ref{3.1})-(\ref{3.4})) are analytic and have faster temporal decay rates than the leading terms. The main difficulty focuses on the interaction of the Riesz operators (the first and second-order) in low frequency, the wave operators and the heat kernels. We merely take a few typical leading terms for instances to elaborate on the avoidance of the seeming singularity at $\xi=0$ and derive different wave patterns in the low frequency based on suitable combinations and cancellations.

To facilitate explanation, we denote $\hat{\mathbf{w}}_t=\cos({\rm c}_1|\xi|t)$ and $\hat{\mathbf{w}}=\frac{\sin({\rm c}_1|\xi|t)}{|\xi|}$ are the Fourier transform of wave operators with the propagation speed ${\rm c}_1$. Then, we first reformulate the following typical term in $\hat{G}_{12}^l$:
\begin{equation}\label{3.5}
\arraycolsep=1.5pt
\begin{array}{rl}
&\frac{e^{\lambda_1t}-e^{\lambda_2t}}{2}\frac{\xi^T}{|\xi|}=\frac{\xi^T}{|\xi|}e^{-b_1|\xi|^2t+\mathcal{O}(|\xi|^4)t}\sin({\rm c}_1|\xi|t+|\xi|\beta(|\xi|^2)t)\\[1mm]
=&\xi^T e^{-b_1|\xi|^2t+\mathcal{O}(|\xi|^4)t}\Big\{\sin({\rm c}_1|\xi|t)\cos(|\xi|\beta(|\xi|^2)t)+\cos({\rm c}_1|\xi|t)\sin(|\xi|\beta(|\xi|^2)t)\Big\}\\[1mm]
=&\xi^T\hat{\mathbf{w}}e^{-b_1|\xi|^2t+\mathcal{O}(|\xi|^4)t}+\xi^T\hat{\mathbf{w}}(\cos(|\xi|\beta(|\xi|^2)t)-1)e^{-b_1|\xi|^2t+\mathcal{O}(|\xi|^4)t}\\[1mm]
&+\xi^T\hat{\mathbf{w}}_t\frac{\sin(|\xi|\beta(|\xi|^2)t)}{|\xi|}e^{-b_1|\xi|^2t+\mathcal{O}(|\xi|^4)t},
\end{array}
\end{equation}
where $\beta(\cdot)$ is an analytic function.

Recall the estimates given in Li \cite{ld} for any $N>0$ that:
\begin{equation}
\arraycolsep=1.5pt
\begin{array}{rl}
&\Big|\mathcal{F}^{-1}\Big((\cos(|\xi|\beta(|\xi|^2)t)-1)e^{-b_1|\xi|^2t+\mathcal{O}(|\xi|^4)t}\Big)\Big|+\Big|\mathcal{F}^{-1}(\frac{\sin(|\xi|\beta(|\xi|^2)t)}{|\xi|}e^{-b_1|\xi|^2t+\mathcal{O}(|\xi|^4)t})\Big|\\[2mm]
\lesssim &(1+t)^{-\frac{3+|\alpha|}{2}}\big(1+\frac{|x|^2}{1+t}\big)^{-N}.
\end{array}
\end{equation}
It together with  Kirchhoff formula in dimension three give that
\begin{equation}\label{3.5(0)}
\Big|\partial_x^\alpha\mathcal{F}^{-1}\Big(\frac{e^{\lambda_1t}-e^{\lambda_2t}}{2}\frac{\xi^T}{|\xi|}\Big)\Big|\lesssim (1+t)^{-\frac{4+|\alpha|}{2}}\Big(1+\frac{(|x|-{\rm c}_1t)^2}{1+t}\Big)^{-N}.
\end{equation}

Similarly, one can immediately see that
\begin{equation}\label{3.5(00)}
\Big|\partial_x^\alpha\mathcal{F}^{-1}\Big(\frac{e^{\lambda_3t}-e^{\lambda_4t}}{2}\frac{\xi^T}{|\xi|}\Big)\Big|\lesssim (1+t)^{-\frac{4+|\alpha|}{2}}\Big(1+\frac{(|x|-{\rm c}_2t)^2}{1+t}\Big)^{-N}.
\end{equation}

Next, we study the most typical entry $G_{22}$, which is different from that of the one-phase fluid model (compressible NS model) in \cite{Ls,Lw}, where
\begin{equation}\label{3.5(1)}
\arraycolsep=1.5pt
\begin{array}{rl}
\hat{G}_{22}^l=&\frac{e^{\lambda_1t}+e^{\lambda_2t}}{2}\frac{\xi\xi^T}{|\xi|^{2}}-e^{-|\xi|^2t}\frac{\xi\xi^T}{|\xi|^{2}}+\cdots\\[2mm]
\sim &\underbrace{(\hat{\mathbf{w}}_t-1)\frac{\xi\xi^T}{|\xi|^{2}}e^{-b_1|\xi|^2t}}_{\text{Riesz\ wave\ I}}+\underbrace{\frac{\xi\xi^T}{|\xi|^{2}}(e^{-b_1|\xi|^2t}-e^{-|\xi|^2t})}_{\text{Riesz\ wave\ I\!I}}+\cdots.
\end{array}
\end{equation}
We use $\hat{R}_i$ with $i=1,2$ to denote the Riesz wave-I and Riesz wave-I\!I, respectively. As in Liu-Noh \cite{Ls} and Li \cite{ld} for the isentropic and non-isentropic compressible Navier-Stokes system, their pointwise space-time descriptions hold for an arbitrarily large integer $N$ that
\begin{align}\label{3.6}
\begin{array}{rl}
&|\partial_x^\alpha(\mathcal{F}^{-1}(\hat{R_1}))|\lesssim (1+t)^{-\frac{3+|\alpha|}{2}}\big(1+\frac{|x|^2}{1+t}\big)^{-\frac{3+|\alpha|}{2}}+(1+t)^{-\frac{4+|\alpha|}{2}}\big(1+\frac{(|x|-{\rm c}t)^2}{1+t}\big)^{-N},\\
&|\partial_x^\alpha(\mathcal{F}^{-1}(\hat{R_2}))|\lesssim (1+t)^{-\frac{3+|\alpha|}{2}}\big(1+\frac{|x|^2}{1+t}\big)^{-N}.
\end{array}
\end{align}

Due to the complicity of the model (\ref{1.1}), its entry $G_{22}$ in low frequency only can be rewritten as
\begin{equation}\label{3.7}
\begin{array}{rl}
\hat{G}_{22}^l=\underbrace{a_1\frac{e^{\lambda_1t}+e^{\lambda_2t}}{2}\frac{\xi\xi^T}{|\xi|^{2}}-a_2e^{-|\xi|^2t}\frac{\xi\xi^T}{|\xi|^{2}}}_{a_1\neq a_2}+\cdots
\sim \underbrace{(\hat{\mathbf{w}}_t-1)\frac{\xi\xi^T}{|\xi|^{2}}e^{-b_1|\xi|^2t}}_{\text{Riesz\ wave\ I}}+\underbrace{\frac{\xi\xi^T}{|\xi|^{2}}e^{-|\xi|^2t}}_{\text{Riesz\ wave\ I\!I\!I}}+\cdots,
\end{array}
\end{equation}
and we also use $\hat{R}_3$ to denote the Riesz wave-I\!I\!I. Although Riesz wave-I\!I\!I is different from Riesz wave-I\!I, we can also get the following according to Corollary \ref{A.3}:
\begin{align}\label{3.8}
\begin{array}{rl}
&|\partial_x^\alpha(\mathcal{F}^{-1}(\hat{R_3}))|\lesssim (1+t)^{-\frac{3+|\alpha|}{2}}\big(1+\frac{|x|^2}{1+t}\big)^{-\frac{3+|\alpha|}{2}}.
\end{array}
\end{align}

In a conclusion, we have the pointwise estimate for Green's function in the low frequency:
\begin{lemma}\label{l 3.1} For $|\alpha|\geq0$, it holds that
\begin{align}\label{3.10}
&when\ (i,j)\neq(2,2),(2,4),(4,2),(4,4),\\
&|\partial_x^\alpha G_{ij}^{l}|
\lesssim (1+t)^{-\frac{4+|\alpha|}{2}}\Big(1+\frac{(|x|-{\rm c}_1t)^2}{1+t}\Big)^{-N}+(1+t)^{-\frac{4+|\alpha|}{2}}\Big(1+\frac{(|x|-{\rm c}_2t)^2}{1+t}\Big)^{-N};\ \ \
\end{align}
\begin{align}\label{3.11}
&when\ (i,j)=(2,2),(2,4),(4,2),(4,4),\\
&\ \ \ \ \ \ \ \ \ \ |\partial_x^\alpha G_{ij}^{l}|\lesssim(1+t)^{-\frac{3+|\alpha|}{2}}\Big(1+\frac{|x|^2}{1+t}\Big)^{-\frac{3+|\alpha|}{2}}+(1+t)^{-\frac{4+|\alpha|}{2}}\Big(1+\frac{(|x|-{\rm c}_1t)^2}{1+t}\Big)^{-N}\\
&\ \ \ \ \ \ \ \ \ \  \ \ \ \ \ \ \ \ \ \ \ \ \ +(1+t)^{-\frac{4+|\alpha|}{2}}\Big(1+\frac{(|x|-{\rm c}_2t)^2}{1+t}\Big)^{-N},
\end{align}
where the constant $N>0$ can be arbitrarily large.
\end{lemma}
We emphasize again that there exist three diffusion waves in the pointwise space-time description of Green's function: a stationary diffusion wave and two moving diffusion waves with different propagation speeds ${\rm c}_1$ and ${\rm c}_2$.

\vspace{3mm}
Combining Lemma \ref{l 3.1}, Remark \ref{r 2.1} and Lemma \ref{l 2.5}, we can conclude the following for each entry of Green's function.
\begin{proposition}{\rm [\textbf{Space-time\ description\ of\ Green's\ function}]}\label{l 3.3}\\
For $|\alpha|\geq0$, one has the following for an arbitrarily large constant $N>0$:
\begin{align}\label{3.17}
&when\ (i,j)\neq(2,2),(2,4),(4,2),(4,4),\ \ \ \nonumber\\
&\ \ \ \ \ \ \  \ |\partial_x^\alpha (G_{ij}-G_S)|
\lesssim (1+t)^{-\frac{4+|\alpha|}{2}}\Big(1+\frac{(|x|-{\rm c}_1t)^2}{1+t}\Big)^{-N}+(1+t)^{-\frac{4+|\alpha|}{2}}\Big(1+\frac{(|x|-{\rm c}_2t)^2}{1+t}\Big)^{-N};
\end{align}
\begin{align}\label{3.11}
&\!\!\!\!\!\!\!\!when\ (i,j)=(2,2),(2,4),(4,2),(4,4),\nonumber\\
&\ \ \ \ \ \ \ \ \ \ \ \ \  |\partial_x^\alpha (G_{ij}-G_S)|\lesssim(1+t)^{-\frac{3+|\alpha|}{2}}\Big(1+\frac{|x|^2}{1+t}\Big)^{-\frac{3+|\alpha|}{2}}+(1+t)^{-\frac{4+|\alpha|}{2}}\Big(1+\frac{(|x|-{\rm c}_1t)^2}{1+t}\Big)^{-N}\nonumber\\
&\ \ \ \ \ \ \ \ \ \  \ \ \ \ \ \ \ \ \ \ \ \ \ \ \ \ \ \ \ \ \ \ \ \ \ \ \ +(1+t)^{-\frac{4+|\alpha|}{2}}\Big(1+\frac{(|x|-{\rm c}_2t)^2}{1+t}\Big)^{-N},
\end{align}
where the singular term $G_S$ is defined in Lemma \ref{l 2.2(1)} and can be regarded as a Dirac-$delta$ function with exponential temporal decay.
\end{proposition}

At the end of this section, it is necessary to emphasize the difficulties concerning Green's function and nonlinear estimates for this truly non-conservative two-phase flow model. The main difficulty is arising from the non-conservative structure of the system (\ref{1.1}) due to two pressure terms and the presence of the Huygens waves in the low frequency of Green's function. In fact, in existing research works on the generalized Huygens principle rely heavily on the nonlinear convolutions in Lemma \ref{A.4}, and these estimates are sharp in the conservative framework (or the divergence form of the nonlinear terms). Here,  the conservative framework includes full conservation laws \cite{Ls,Lw}, a partially conservation laws based on a specific compensatory cancellation in Green’s function or the nonlinear terms as in \cite{Wu4,Wu5,Wu6}. As for the non-conservative model studied in the present paper, one can hardly find two kinds of linear combination to compensate this non-conservation as usual. Therefore,  to achieve the relatively sharp pointwise space-time description of the solution to the ``genuinely" non-conservative model (\ref{1.1}), we have to construct a series of new nonlinear convolution estimates of different wave patterns, especially for the nonlinear coupling of two Huygens waves with different propagation speeds. See the details in Lemma \ref{l 4.2}.

\section{Pointwise estimates for nonlinear system}

\quad\quad In this section, by using the representation of the solution, we derive the pointwise estimates for the solution of the nonlinear system based on the pointwise space-time description of Green's function and the convolution estimates of nonlinear coupling.

Before the process of the proof, we shall state some useful estimates for later use. The first lemma is used to deal with initial propagation.

\begin{lemma}\label{l 4.1}
There exists a constant $C>0$, such that
\begin{equation*}
\arraycolsep=1.5pt
\begin{array}{ll}
\mathcal{I}_1:=&\int_{\mathbb{R}^3}\left(1+\frac{|x-y|^2}{1+t}\right)^{-\frac{3}{2}}(1+|y|^2)^{-r_1}dy\leq C\left(1+\frac{|x|^2}{1+t}\right)^{-\frac{3}{2}},\ {\rm for}\ r_1>\frac{3}{2};\\
\mathcal{I}_2:=&\int_{\mathbb{R}^3} \bigg(1+\frac{(|x-y|-{\rm c}t)^2}{1+t}\bigg)^{-N}(1+|y|^2)^{-r_2}dy\leq C\left(1+\frac{(|x|\!-\!{\rm c}t)^2}{1+t}\right)^{-1},\ {\rm for}\ r_2>\frac{19}{10},
\end{array}
\end{equation*}
where the positive constant $N$ is suitably large.
\end{lemma}
\begin{proof}
The estimates $\mathcal{I}_1$ is used to deal with the initial propagation for the convolution of $Riesz\ wave$ respectively, and the proof of them can be seen in \cite{Wu5}. Hence, we only prove $\mathcal{I}_2$. When $(|x|-{\rm c}t)^2\leq 4(1+t)$,
$$
\mathcal{I}_2\leq C\leq C\left(1+\frac{(|x|-{\rm c}t)^2}{1+t}\right)^{-r_1}.
$$
When $|x|-{\rm c}t\geq 2\sqrt{1+t}$, we break integration into two parts. If $|y|\geq\frac{|x|-{\rm c}t}{2}$, then
\begin{equation*}
\arraycolsep=1.5pt
\begin{array}{rl}
\mathcal{I}_2\leq & C(1+(|x|-{\rm c}t)^2)^{-1}\int_{\mathbb{R}^3} \bigg(1+\frac{(|x-y|-{\rm c}t)^2}{1+t}\bigg)^{-N}(1+|y|^2)^{-(r_2-1)}dy\\
\leq & C(1+(|x|-{\rm c}t)^2)^{-1}(1+t)\left(\int_{\mathbb{R}^3}(1+|y|^2)^{-\frac{5}{3}(r_2-1)}dy\right)^{\frac{3}{5}}\\
\leq & C\left(1+\frac{(|x|\!-\!{\rm c}t)^2}{1+t}\right)^{-1},\ {\rm for}\ r_2>\frac{19}{10}.
\end{array}
\end{equation*}
Here we have used Young inequality and the following fact in \cite{Lw}:
$$
\int_{\mathbb{R}^n}\left(1+\frac{(|y|-a)^2}{b}\right)^{-r_2}dy\leq C(b^{\frac{n}{2}}+b^{\frac{1}{2}}a^{n-1}),\ {\rm for}\ r_2>\frac{n}{2}.
$$
If $|y|<\frac{|x|-{\rm c}t}{2}$, then $|x-y|-{\rm c}t\geq\frac{||x|-{\rm c}t|}{2}$. Thus, it holds that
\begin{equation*}
\mathcal{I}_2\leq C\bigg(1+\frac{(|x|-{\rm c}t)^2}{1+t}\bigg)^{\!-N/2}\int_{\mathbb{R}^3}\!\!\bigg(1+\frac{(|x-y|-{\rm c}t)^2}{1+t}\bigg)^{\!-N/2}\!\!(1+|y|^2)^{-r_1}dy
\leq C\left(1+\frac{(|x|\!-\!{\rm c}t)^2}{1+t}\right)^{\!-\frac{N}{2}}\!\!.
\end{equation*}
The case: ${\rm c}t-|x|\geq 2\sqrt{1+t}$ can be similarly treated as the case: $|x|-{\rm c}t\geq 2\sqrt{1+t}$. Then, we have completed the proof.
\end{proof}



Without loss of generality, we use $\partial_x^k$ to denote $\partial_x^\alpha$ with $|\alpha|=k$.
By using Duhamel's principle, we can get the representation of the solution $(n^+,\mathbf{m}^+,n^-,\mathbf{m}^-)$ for the nonlinear problem (\ref{2.1}):
\begin{eqnarray}\label{4.1}
 &&\partial_x^k(n^+,\mathbf{m}^+,n^-,\mathbf{m}^-)^T\nonumber\\
 &=&\partial_x^k G(x,t)\ast_x(n^+_0,\mathbf{m}^+_0,n^-_0,\mathbf{m}^-_0)^T
          +\int_0^t\partial_x^k G(\cdot,t-\tau)\ast_x(0,F_1,0,F_2)^T(\cdot,\tau)d\tau,
\end{eqnarray}
where the nonlinear terms $F_1$ and $F_2$ are defined in (\ref{2.2}).

\textbf{Initial propagation.} Use $(\breve{n}^+,\breve{\mathbf{m}}^+,\breve{n}^-,\breve{\mathbf{m}}^-)$ to denote the linear part of the solution in (\ref{4.1}). According to the representation (\ref{4.1}), Proposition \ref{l 3.3}, the initial condition (\ref{1.16}) and  Lemma \ref{l 4.1} for the initial propagation, we can immediately get the following:
\begin{equation}\label{4.2}
\begin{array}{rl}
 &\left|\partial_x^k\left(\!\!
            \begin{array}{ccc}
              \breve{n}^+ \\
              \breve{n}^-
            \end{array}
          \!\!\right)\right|=\!\left|\partial_x^k\!\left(\!
                            \begin{array}{cccc}
                              G_{11} & G_{12} & G_{13} &G_{14}\\
                               G_{31} & G_{32} & G_{33} &G_{34}\\
                            \end{array}
                          \!\right)\!\!\ast_x(n^+_0,\mathbf{m}^+_0,n^-_0,\mathbf{m}^-_0)^T
          \right|\\[3mm]
          \leq\!\! & \left|\partial_x^k\!\bigg(\!
                            \begin{array}{cccc}
                             G_{11}-G_{S} & G_{12}-G_{S} & G_{13}-G_{S} & G_{14}-G_{S}\\
                               G_{31}-G_{S} & G_{32}-G_{S} & G_{33}-G_{S} & G_{34}-G_{S}\\
                            \end{array}
                          \!\bigg)\!\!\ast_x (n^+_0,\mathbf{m}^+_0,n^-_0,\mathbf{m}^-_0)^T\right|\\[3mm]
          &+\left| \left(\!
                            \begin{array}{cccc}
                              G_{S} &G_{S} & G_{S} &  G_{S}\\
                              G_{S} & G_{S} & G_{S} & G_{S}\\
                            \end{array}
                          \!\right)\!\ast_x\partial_x^k (n^+_0,\mathbf{m}^+_0,n^-_0,\mathbf{m}^-_0)^T\right|\\[3mm]
          \leq\!\! &\displaystyle C\epsilon\bigg((1+t)^{-2}\Big(1+\frac{(|x|-{\rm c}_1t)^2}{1+t}\Big)^{-1}+(1+t)^{-2}\Big(1+\frac{(|x|-{\rm c}_2t)^2}{1+t}\Big)^{-1}\bigg),\ \ k=0,1,
\end{array}
\end{equation}
\begin{equation}\label{4.3}
\begin{array}{rl}
 &\!\!\!\!\!\!\left|\partial_x^k\left(\!\!
            \begin{array}{ccc}
              \breve{\mathbf{m}}^+ \\
              \breve{\mathbf{m}}^-
            \end{array}
          \!\!\right)\right|=\!\left|\partial_x^k\!\left(\!
                            \begin{array}{cccc}
                              G_{21} & G_{22} & G_{23} &G_{24}\\
                               G_{41} & G_{42} & G_{43} &G_{44}\\
                            \end{array}
                          \!\right)\!\!\ast_x (n^+_0,\mathbf{m}^+_0,n^-_0,\mathbf{m}^-_0)^T\right|\\[3mm]
          \!\!\!\!\leq\!\! & \left|\partial_x^k\!\left(\!
                            \begin{array}{cccc}
                             G_{21}-G_{S} & G_{22}-G_{S} & G_{23}-G_{S} &G_{24}-G_{S}\\
                               G_{41}-G_{S} & G_{42}-G_{S} & G_{43}-G_{S} &G_{44}-G_{S}\\
                            \end{array}
                          \!\right)\!\!\ast_x (n^+_0,\mathbf{m}^+_0,n^-_0,\mathbf{m}^-_0)^T\right|\\[3mm]
          &+\left| \bigg(\!
                            \begin{array}{cccc}
                              G_{S} & G_{S} & G_{S} & G_{S}\\
                              G_{S} & G_{S} & G_{S2} &G_{S}
                            \end{array}
                          \!\bigg)\!\ast_x \partial_x^k (n^+_0,\mathbf{m}^+_0,n^-_0,\mathbf{m}^-_0)^T\right|\\[3mm]
          \!\!\!\!\leq\!\! &\!\!\displaystyle C\epsilon\bigg((1+t)^{-\frac{3}{2}}\Big(1+\frac{|x|^2}{1+t}\Big)^{-\frac{3}{2}}
          +(1+t)^{-2}\Big(1+\frac{(|x|-{\rm c}_1t)^2}{1+t}\Big)^{-1}+(1+t)^{-2}\Big(1+\frac{(|x|-{\rm c}_2t)^2}{1+t}\Big)^{-1}\bigg).
\end{array}
\end{equation}

\begin{remark}It seems that we can only obtain the corresponding estimate by replacing the decay $(1+t)^{-a}$ in (\ref{4.2})-(\ref{4.3}) with $t^{-a}$, since $G_{S}$ arising from the high frequency of Green's function are singular at $t=0$. Nevertheless, similar to the study on the initial propagation of the one-phase fluid model in Proposition 6.3 of \cite{Ls} by considering two subcases $0\leq t<1$ and $t\geq1$, one can easily achieve the above two estimates of initial propagation for the two-phase fluid model when $t\geq0$. We omit the details here since it has no essential impact on the conclusion.
\end{remark}

\textbf{Nonlinear Coupling.} According to the above initial propagation and the nonlinear convolution estimates established in Lemma \ref{l 4.2}, we give the ansatz on the solution to the nonlinear problem when $k\leq1$:
\begin{equation}\label{4.4}
|\partial_x^k(R^\pm-1)|\lesssim (1+t)^{-2}\Big(1+\frac{|x|^2}{1+t}\Big)^{-1}+(1+t)^{-2}\Big(1+\frac{(|x|\!-\!{\rm c}_1t)^2}{1+t}\Big)^{-1}+(1+t)^{-2}\Big(1+\frac{(|x|\!-\!{\rm c}_2t)^2}{1+t}\Big)^{-1},
\end{equation}
\begin{equation}\label{4.5}
|\partial_x^k\mathbf{m}^\pm|\lesssim (1+t)^{-\frac{3}{2}}\Big(1+\frac{|x|^2}{1+t}\Big)^{-\frac{3}{2}}+(1+t)^{-2}\Big(1+\frac{(|x|\!-\!{\rm c}_1t)^2}{1+t}\Big)^{-1}+(1+t)^{-2}\Big(1+\frac{(|x|\!-\!{\rm c}_2t)^2}{1+t}\Big)^{-1}.
\end{equation}

Due to the non-conservation of the system (\ref{1.1}), and the slower decay rate of the two fraction densities of the system (\ref{1.1}), we have to construct some refined nonlinear convolution estimates, especially for the convolution of Riesz\ wave (or diffusion wave) and Huygens\ wave-I (or Huygens wave-I\!I),  and the convolution of Huygens\ wave-I and Huygens wave-I\!I. In fact, partially inspired by the idea in \cite{lin1,lin2}, we have
\begin{lemma}\label{l 4.2} It holds that

{\rm (}``Huygens\ wave"\ convolved\ with\ ``diffusion\ wave"\ {\rm{)}}
\begin{align}
\mathcal{K}_1=&\int_0^t\!\int_{\mathbb{R}^3}\underline{(1+t-\tau)^{-2}\Big(1+\frac{(|x-y|-{\rm c}(t-\tau))^2}{1+t-\tau}\Big)^{-N}}(1+\tau)^{-3}\Big(1+\frac{|y|^2}{1+\tau}\Big)^{-3}\!\!dyd\tau\nonumber\\[3mm]
\lesssim &\ \underbrace{(1+t)^{-\frac{3}{2}}\Big(1+\frac{|x|^2}{1+t}\Big)^{-\frac{3}{2}}}_{ {\rm or},\ \ (1+t)^{-2}\big(1+\frac{|x|^2}{1+t}\big)^{-1}}+(1+t)^{-2}\Big(1+\frac{(|x|-{\rm c}t)^2}{1+t}\Big)^{-1},\label{4.6}
\end{align}
\noindent {\rm (}``Riesz\ wave"\ convolved\ with\ ``Huygens\ wave"\ {\rm{)}}
\begin{align}
\mathcal{K}_2=&\int_0^t\!\int_{\mathbb{R}^3}(1+t-\tau)^{-\frac{3}{2}}\Big(1+\frac{|x-y|^2}{1+t-\tau}\Big)^{-\frac{3}{2}}(1+\tau)^{-4}\Big(1+\frac{(|y|-{\rm c}\tau)^2}{1+\tau}\Big)^{-2}dyd\tau\nonumber\\[3mm]
\lesssim &\ (1+t)^{-\frac{3}{2}}\Big(1+\frac{|x|^2}{1+t}\Big)^{-\frac{3}{2}}+(1+t)^{-2}\Big(1+\frac{(|x|-{\rm c}t)^2}{1+t}\Big)^{-1},\label{4.7}
\end{align}
{\rm (}``Huygens\ wave\ I"\ convolved\ with\ ``Huygens\ wave\ II"\ {\rm{)}}
\begin{align}
\mathcal{K}_3=&\int_{0}^{t}\!\int_{\mathbb{R}^3}\underline{(1+t-\tau)^{-2}\Big(1+\frac{(|x-y|-{\rm c}_1(t-\tau))^2}{1+t-\tau}\Big)^{-N}}(1+\tau)^{-4}\Big(1+\frac{(|y|-{\rm c}_2\tau)^2}{1+\tau}\Big)^{-2}dyd\tau\nonumber\\[3mm]
\lesssim &\ \underbrace{(1+t)^{-\frac{3}{2}}\Big(1+\frac{|x|^2}{1+t}\Big)^{-\frac{3}{2}}}_{ {\rm or},\ \ (1+t)^{-2}\big(1+\frac{|x|^2}{1+t}\big)^{-1}}+(1+t)^{-2}\Big(1+\frac{(|x|-{\rm c}_1t)^2}{1+t}\Big)^{-1}+(1+t)^{-2}\Big(1+\frac{(|x|-{\rm c}_2t)^2}{1+t}\Big)^{-1}.\label{4.8}
\end{align}
Here the constant $N>0$ can be arbitrarily large, and the positive constants ${\rm c},{\rm c}_1,{\rm c}_2$ denote the propagation speed of the Huygens waves.
\end{lemma}

\begin{remark}Note that the underlined two terms in (\ref{4.6}) and (\ref{4.8}) are different from the previous nonlinear estimates given in \cite{Ls,Lw}, where an additional $(1+t)^{-\frac{1}{2}}$-decay rate is from the conservative structure of the single-phase model: compressible Navier-Stokes equations. In fact, the conservative structure is rather important in \cite{Ls,Lw}, since the decay rate of the Huygens wave in the $L^p$ framework with $p<2$ is slower than that of  the heat kernel. 
Therefore, to overcome the difficulties induced by the non-conservative nature, we shall revisit this nonlinear convolution estimate via a refined space-time decomposition argument together with fully exploiting the interchange between temporal and spatial variables within specific regions, thereby deriving satisfactory pointwise space-time estimates.
\end{remark}

Owing to its complexity, the detailed proof of Lemma \ref{l 4.2} (especially the proof of (\ref{4.8})) is deferred to the Appendix.

With these nonlinear estimates in hand, we shall close the ansatz on the solution of the nonlinear problem by substituting (\ref{4.4})-(\ref{4.5}) into the representation of the solution $(n^+,\mathbf{m}^+,n^-,\mathbf{m}^-)$ in (\ref{4.1})-(\ref{4.2}) and using the aforementioned nonlinear convolution estimates.  In particular, we have to split Green's function into the regular term $G_{ij}-G_{S}$ and the singular term $G_{S}$. Then, one can put all of the derivatives on $G_{ij}-G_{S}$ to deal with the convolution between the regular part $G_{ij}-G_{S}$ and the nonlinear terms, however, one have to put all of the derivatives on the nonlinear terms to deal with the convolution of the singular part and the nonlinear terms.

We just take nonlinear coupling of the momenta $\mathbf{m}^\pm$ for example, since there exist the additional Riesz waves in the related entries of Green's function. To facilitate the description, we use $\tilde{\mathbf{m}}^\pm$ to denote the nonlinear part of $\mathbf{m}^\pm$ in (\ref{4.5}). Then by using Proposition \ref{l 3.3} and the nonlinear convolution estimates in Lemma \ref{l 4.2}, one has for $0\leq k\leq1$ that
\begin{align}\label{5.12}
|\partial_x^k\tilde{\mathbf{m}}^+|\leq &\int_0^t|\partial_x^k(G_{22}-G_{S})(\cdot,t-\tau)\ast F_1(\cdot,\tau)|d\tau+\int_0^t|\partial_x^k(G_{24}-G_{S})(\cdot,t-\tau)\ast F_2(\cdot,\tau)|d\tau\nonumber\\
&+\int_0^t|G_{S}(\cdot,t-\tau)\ast \partial_x^k (F_1(\cdot,\tau),F_2(\cdot,\tau))|d\tau\nonumber\\
\leq &\ 2C\epsilon\bigg((1+t)^{-\frac{3}{2}}\Big(1+\frac{|x|^2}{1+t}\Big)^{-\frac{3}{2}}+(1+t)^{-2}\Big(1+\frac{(|x|\!-\!{\rm c}_1t)^2}{1+t}\Big)^{-1}+(1+t)^{-2}\Big(1+\frac{(|x|\!-\!{\rm c}_2t)^2}{1+t}\Big)^{-1}\bigg)\nonumber\\
      &+\int_0^t|G_{S}(\cdot,t-\tau)\ast \partial_x^k (\partial_x(n^+\partial_x\mathbf{m}^++n^-\partial_x\mathbf{m}^-)+\cdots)(\cdot,\tau)|d\tau\nonumber.
\end{align}
Obviously, when $k=0$, we have to use the pointwise ansatz of the unknowns $(n^\pm,\mathbf{m}^\pm)$ and its first derivative. Additionally, when $k=1$, it further requires the $L^\infty$-estimate of  $\partial_x^3\mathbf{m}^\pm$ when dealing with the singular part $G_S$ (note that $G_S$ is like a delta function with exponential temporal decay rate) convolved with the nonlinear terms. Thus, we can close the ansatz (\ref{4.4})-(\ref{4.5}) in $H^5$-framework based on the Sobolev embedding theorem, and then complete the proof of Theorem \ref{l 1.1}.

\section{Appendix: Analytic tools}\label{1section_appendix}

\subsection{Some useful lemmas}

The first one is used to derive the pointwise estimates of Green's function in the low frequency.


\begin{lemma}\label{A.1}[Wu-Wang\cite{Wu4}] If there exists a
constant $C>0$ such that when $|\xi|\leq1$, $\hat{f}(\xi,t)$ satisfies
$$
|D_\xi^\beta(\xi^\alpha\hat{f}(\xi,t))|\leq
C(|\xi|^{(|\alpha|-|\beta|)_+}
+|\xi|^{|\alpha|}t^{|\beta|/2})(1+(t|\xi|^2))^a \exp(-b|\xi|^2t),
$$
for some constant $b>0$ and any multi-indexes $\alpha, \beta$ with $|\beta|\leq 2N$, then
\begin{equation}\label{7.1}
|D_x^\alpha f(x,t)|\leq C_N (1+t)^{-(n+|\alpha|)}B_N(|x|,t),
\end{equation}
where $a$ is any fixed integer, $(e)_+=\max(0,e)$ and
$$
B_N(|x|,t)=\Big(1+\frac{|x|^2}{1+t}\Big)^{-N}.
$$
 \end{lemma}

The following lemma is used to deal with the terms containing double Riesz operator with the symbol $\frac{\xi\xi^T}{|\xi|^2}$ in low frequency part of Green's function. For readers' convenience, we write down the proof since it was just stated in the previous works without a detailed proof.

\begin{lemma}\label{A.2}\cite{Wu4}
Let $t>0$ and $x\in\mathbb{R}^n$ with $n\geq2$. Suppose that $f(x,t)$ satisfies
 \begin{equation}\label{8.10}
 \begin{array}{rl}
&\ \ \ \ \ |f(x,t)|\lesssim \big(1+\frac{|x|^2}{1+t}\big)^{-r_1},\ \ \ \ \ \ \ \ \ r_1>\frac{n}{2},\\
&|\nabla f(x,t)|\lesssim (1+t)^{-\frac{1}{2}}\big(1+\frac{|x|^2}{1+t}\big)^{-r_2},\ \ \ r_2>\frac{n+1}{2}.
\end{array}
\end{equation}
Then it holds that
 \begin{equation}\label{8.11}
\Big|\nabla{\rm div}\Delta^{-1}f(x,t)\Big|\lesssim  \Big(1+\frac{|x|^2}{1+t}\Big)^{-\frac{n}{2}}.
\end{equation}
\end{lemma}


\begin{corollary}\label{A.3}\cite{Wu4}
Suppose that $\hat{f}(\xi,t)=\frac{\xi\xi^T}{|\xi|^2}\chi_1(\xi)e^{-a|\xi|^2t+\mathcal{O}(|\xi|^3)t}$ in the Fourier space, where $a$ is a positive constant and $\chi_1(\xi)$ is the cutoff function for the low frequency.   Then, it holds that
\begin{equation}\label{8.18}
|\mathcal{F}^{-1}(\xi^\alpha\hat{f}(\xi,t))|\lesssim (1+t)^{-\frac{n+|\alpha|}{2}}\big(1+\frac{|x|^2}{1+t}\big)^{-\frac{n+|\alpha|}{2}}.
\end{equation}
\end{corollary}

 The next one describes the singular part of the high frequency:
\begin{lemma}\label{A.4}[Wang-Yang\cite{Wang}]
If ${\rm supp}\hat{f}(\xi)\subset
O_K=:{\{\xi, |\xi|\geq K>0\}}$, and $\hat{f}(\xi)$ satisfies
\begin{equation*}
|D_\xi^\beta\hat{f}(\xi)|\leq C|\xi|^{-|\beta|-1},
\end{equation*}
then there exist distributions $f_1(x), f_2(x)$ and a constant $C_0$
such that $$
f(x)=f_1(x)+\underbrace{f_{2}(x)+C_{0}\delta(x)}_{G_{S1}}
$$
where $\delta(x)$ is the Dirac function. Furthermore, for any $|\alpha|\geq0$ and any positive integer $N$, we have
\begin{equation*}
|D_x^\alpha f_1(x)|\leq C(1+|x|^2)^{-N},\
\|f_{2}\|_{L^1}\leq C,\ {\rm supp}f_{2}(x)\subset\{x;|x|<\eta_0\ll1\}.
\end{equation*}
\end{lemma}

\begin{lemma}\label{A.5}[Liu-Noh\cite{Ls},Liu-Wang\cite{Lw}] There exists a constant $C>0$ such that
\begin{equation*}\label{6.1}
\begin{array}{ll}
\displaystyle K_1=\!\int_0^t\!\!\int_{\mathbb{R}^3}(1+t-\tau)^{-2}\Big(1+\frac{|x-y|^2}{1\!+\!t\!-\!s}\Big)^{-2}(1+\tau)^{-3}\Big(1+\frac{|y|^2}{1\!+\!\tau}\Big)^{-3}\!\!dyd\tau
\leq C(1+t)^{-2}\Big(1+\frac{|x|^2}{1+t}\Big)^{-\frac{3}{2}},\\[3.5mm]
\displaystyle K_2=\!\int_0^t\!\!\int_{\mathbb{R}^3}(1+t-\tau)^{-2}\Big(1+\frac{|x-y|^2}{1+t-\tau}\Big)^{-2}(1+\tau)^{-4}\Big(1+\frac{(|y|-{\rm c}\tau)^2}{1+\tau}\Big)^{-3}dyd\tau\\[2mm]
\ \ \ \ \leq \ \displaystyle C(1+t)^{-2}\Big(\Big(1+\frac{|x|^2}{1+t}\Big)^{-\frac{3}{2}}+\Big(1+\frac{(|x|-{\rm c}t)^2}{1+t}\Big)^{-\frac{3}{2}}\Big),\\[2mm]
K_3=\displaystyle\! \int_{0}^{t}\!\!\int_{\mathbb{R}^3}(1+t-\tau)^{-\frac{5}{2}}\Big(1+\frac{(|x-y|-{\rm c}(t-\tau))^2}{1\!+\!t\!-\!s}\Big)^{-N}(1+\tau)^{-4}\Big(1+\frac{(|y|-{\rm c}\tau)^2}{1\!+\!s}\Big)^{-3}dyd\tau\\[2mm]
\ \ \ \ \leq \ \displaystyle C(1+t)^{-2}\Big(\big(1+\frac{|x|^2}{1+t}\big)^{-\frac{3}{2}}+\Big(1+\frac{(|x|-{\rm c}t)^2}{1+t}\Big)^{-\frac{3}{2}}\Big),
\end{array}
\end{equation*}
where the constant $N>0$ can be arbitrarily large.
\end{lemma}
\begin{remark}These three nonlinear estimates are mainly used to the conservative compressible fluid models, such as the compressible Navier-Stokes equations in \cite{Ls,Lw}, and the equivalent conservative compressible fluid models, such as the compressible bipolar Navier-Stokes-Poisson equations in \cite{Wu4,Wu5}. Here, the ``equivalent conservation property" implies that when dealing with nonlinear convolution estimates for nonlinear compressible fluid models, these non-conservative fluid models can be transformed into structures analogous to conservative fluid models through cancellation effects in either their linear or nonlinear components.
\end{remark}

Finally, we shall state the other two lemmas, which will be used to prove the nonlinear convolution estimates in Lemma \ref{l 4.2}.


\begin{lemma}\label{A.7}\cite{Wang2} For $0\leq \tau\leq t$, $A^2\geq 1+t$ and the constant $a>0$, then
\begin{equation*}
\bigg(1+\frac{A^2}{1+t}\bigg)^{-a}\leq 3^a\bigg(\frac{1+\tau}{1+t}\bigg)^a\bigg(1+\frac{A^2}{1+t}\bigg)^{-a}.
\end{equation*}
\end{lemma}

\begin{lemma}\label{A.8}
If $|x|\leq Ct$ for some constant $C>0$, $N\geq3$ and ${\rm c}_1,{\rm c}_2>0$, then it holds that
\begin{align*}\label{9.2}
\int_0^t\int_{\mathbb{R}^3}\bigg(1+\frac{(|x-y|-{\rm c}_1(t-\tau))^2}{1+t}\bigg)^{-N}\bigg(1+\frac{(|y|-{\rm c}_2\tau)^2}{1+t}\bigg)^{-2}dyd\tau
\lesssim (1+t)^3.
\end{align*}
\end{lemma}
\begin{proof} In fact, when $\big(1+\frac{(|y|-{\rm c}_2\tau)^2}{1+t}\big)^{-2}$ was replaced by $\big(1+\frac{(|y|-{\rm c}_2\tau)^2}{1+t}\big)^{-3}$, it has been proved in Bai-Zhang \cite{Bai}. As for our case, one can also get the same result. The details are omitted for simplicity.
\end{proof}

\subsection{Proof of Lemma \ref{l 4.2}}

\quad\quad We will prove the essential estimates of nonlinear coupling: Lemma \ref{l 4.2}.

At first, to facilitate the estimate $\mathcal{K}_3$ in Lemma \ref{l 4.2}, we decompose space-time domain into the following eight regions:
\begin{align*}
D_{1}& =\bigg \{ |x|\leq \sqrt{1+t}\bigg \}\cup\left \{ \sqrt{1+t}\leq|x|\leq \frac{{\rm c}_1t}{2}\right \}:=D_{11}\cup D_{12}, \\
\ D_{2}& =\left \{ \frac{{\rm c}_1t}{2}\leq|x|\leq \frac{({\rm c}_1+{\rm c}_2)t}{2}\right \}\\
&=\underbrace{\bigg \{ \big||x|-{\rm c}_1t\big|\leq \sqrt{1+t}\bigg \}}_{D_{21}}\cup\underbrace{\left \{ \frac{{\rm c}_1t}{2}\leq|x|\leq {\rm c}_1t-\sqrt{1+t}\right \}}_{D_{22}}
\cup\underbrace{\left \{ {\rm c}_1t+\sqrt{1+t}\leq|x|\leq \frac{({\rm c}_1+{\rm c}_2)t}{2}\right \}}_{D_{23}}, \\
D_{3}& =\underbrace{\bigg \{ \big||x|-{\rm c}_2t\big|\leq \sqrt{1+t}\bigg \}}_{D_{31}}\cup\underbrace{\left\{ \frac{({\rm c}_1+{\rm c}_2)t}{2}\leq|x|\leq {\rm c}_2t-\sqrt{1+t}\right\}}_{D_{32}}
\cup\underbrace{\bigg\{|x|\geq  {\rm c}_2t+\sqrt{1+t} \bigg\}}_{D_{33}}.
\end{align*}
Here $0<{\rm c}_1\leq {\rm c}_2$ and $t$ is suitably large (since the nonlinear convolution estimates can be easily established when $t$ is upper bounded).

\vspace{3mm}
\textit{\textbf{Proof of $\mathcal{K}_3$ in Proposition \ref{l 4.2}}}. (Interaction of Huygens waves for {\bf{non-conservative system}})


\vspace{3mm}

\noindent\textbf{Case 1:} $\left( x,t\right) \in D_{11}\cup D_{21}\cup D_{31}$. Direct computation
gives for ${\rm c}={\rm c}_1$ or ${\rm c}={\rm c}_2$ that
\begin{eqnarray*}
\arraycolsep=1.5pt
&&\int_{0}^{\frac{t}{2}}\int_{\mathbb{R}^{3}}\left( 1+t\right)
^{-2}\left( 1+\tau \right) ^{-4}\left( 1+\frac{\left( \left \vert y\right
\vert -{\rm c}\tau \right) ^{2}}{1+\tau }\right) ^{-2}dyd\tau \\
&&+\int_{\frac{t}{2}}^{t}\int_{\mathbb{R}^{3}}\left( 1+t-\tau \right)
^{-2}\Big(1+\frac{(|x-y|-{\rm c}(t-\tau))^2}{1+t-\tau}\Big)^{-N}\left( 1+t\right)
^{-4}dyd\tau \\
&\lesssim &\left( 1+t\right) ^{-2}\int_{0}^{\frac{t}{2}}\left( 1+\tau
\right) ^{-4}\left( 1+\tau \right) ^{\frac{5}{2}}d\tau +\left( 1+t\right)
^{-4}\int_{\frac{t}{2}}^{t}\left( 1+t-\tau \right) ^{-2}\left( 1+t-\tau
\right) ^{\frac{5}{2}}d\tau \\
&\lesssim &(1+t)^{-2}\hbox{,}
\end{eqnarray*}
which further implies for any constant $N>0$ that
\begin{eqnarray*}
\mathcal{K}_3 \lesssim ( 1+t)^{-2}\bigg(1+\frac{|x|^2}{1+t}\bigg)^{-N}\hbox{,}\ \ \ \ \ \ {\rm in}\ D_{11},
\end{eqnarray*}
\begin{eqnarray*}
\mathcal{K}_3 \lesssim ( 1+t)^{-2}\bigg(1+\frac{(|x|-{\rm c}_1t)^2}{1+t}\bigg)^{-N}\hbox{,}\ \ \ \ \ {\rm in}\ D_{21},
\end{eqnarray*}
\begin{eqnarray*}
\mathcal{K}_3 \lesssim ( 1+t)^{-2}\bigg(1+\frac{(|x|-{\rm c}_2t)^2}{1+t}\bigg)^{-N}\hbox{,}\ \ \ \ \ {\rm in}\ D_{31}.
\end{eqnarray*}

\noindent\textbf{Case 2:} $\left( x,t\right) \in D_{12}=\left \{ \sqrt{1+t}\leq|x|\leq \frac{{\rm c}_1t}{2}\right \}$. We are going to study this case by dividing the time interval into four parts.

\vspace{1.5mm}
\textbf{Case 2.1:} $0\leq \tau\leq\frac{{\rm c}_1t-|x|}{2({\rm c}_1+{\rm c}_2)}$. When $|y|\geq\frac{{\rm c}_1t-|x|+({\rm c}_2-{\rm c}_1)\tau}{2}$, which implies
$$
|y|-{\rm c}_2\tau\geq\frac{{\rm c}_1t-|x|}{2}-\frac{({\rm c}_2+{\rm c}_1)\tau}{2}\geq \frac{{\rm c}_1t-|x|}{2}-\frac{{\rm c}_1t-|x|}{4}\geq \frac{{\rm c}_1t-|x|}{4},
$$
and when $|y|\leq\frac{{\rm c}_1t-|x|+({\rm c}_2-{\rm c}_1)\tau}{2}$, which implies that
$$
{\rm c}_1(t-\tau)-|x-y|\geq {\rm c}_1t-|x|-{\rm c}_1\tau-\frac{{\rm c}_1t-|x|+({\rm c}_2-{\rm c}_1)\tau}{2}\geq \frac{{\rm c}_1t-|x|}{4}.
$$
Then, by using Lemma \ref{A.7}, the facts $({\rm c}_1t-|x|)^2\geq(\frac{{\rm c}_1t}{2})^2\geq 1+t$ in Case 2.1 and 
\begin{equation}\label{9.0}
\int_{\mathbb{R}^3}\Big(1+\frac{(|x|-{\rm c}t)^2}{1+t}\Big)^{-r}dx\lesssim (1+t)^{\frac{5}{2}}\ {\rm for\ constants}\ r>\frac{3}{2}\ {\rm and}\ {\rm c}\neq0,
\end{equation}
one has
 \begin{eqnarray*}
 \arraycolsep=1.5pt
&&J_1\triangleq\int_{0}^{\frac{{\rm c}_1t-|x|}{2({\rm c}_1+{\rm c}_2)}}\left( \int_{|y|\geq\frac{{\rm c}_1t-|x|+({\rm c}_2-{\rm c}_1)\tau}{2}}+\int_{|y|\leq\frac{{\rm c}_1t-|x|+({\rm c}_2-{\rm c}_1)\tau}{2}}\right) \left( \cdots \right) dyd\tau\\
&\lesssim\!\!&\!\! (1+t)^{-1}\Big(1+\frac{(|x|-{\rm c}_1t)^2}{1+t}\Big)^{-1}\!\int_{0}^{t}\!\!\int_{|y|\geq\frac{{\rm c}_1t-|x|+({\rm c}_2-{\rm c}_1)\tau}{2}}(1+t-\tau)^{-1}\Big(1+\frac{(|x-y|-{\rm c}_1(t-\tau))^2}{1+t-\tau}\Big)^{-N}\\
&&\ \ \ \ \ \ \  \ \ \ \ \ \ \ \ \ \ \ \  \ \ \ \ \ \ \ \ \ \ \ \  \ \ \ \ \ \ \ \ \ \ \ \  \ \ \ \ \ \ \ \ \ \ \ \cdot(1+\tau)^{-4}\bigg(\frac{1+\tau}{1+t}\bigg)\Big(1+\frac{(|y|-{\rm c}_2\tau)^2}{1+\tau}\Big)^{-1}dyd\tau\\
&&\!\!+(1+t)^{-2}\Big(1+\frac{(|x|-{\rm c}_1t)^2}{1+t}\Big)^{-N}\int_{0}^{t}\int_{|y|\leq\frac{{\rm c}_1t-|x|+({\rm c}_2-{\rm c}_1)\tau}{2}}(1+\tau)^{-4}\Big(1+\frac{(|y|-{\rm c}_2\tau)^2}{1+\tau}\Big)^{-2}\bigg(\frac{1+t-\tau}{1+t}\bigg)^Ndyd\tau\\
&\lesssim&\!\! (1+t)^{-2}\Big(1+\frac{(|x|-{\rm c}_1t)^2}{1+t}\Big)^{-1}\int_{0}^{t}\underbrace{(1+t-\tau)^{-1}(1+t-\tau)^{\frac{5}{2}\cdot\frac{2}{5}}(1+\tau)^{-3}(1+\tau)^{\frac{5}{2}\cdot \frac{3}{5}}}_{where\ we\ have\ used\ Young\ inequality\ with\ p=\frac{5}{2}\ and\ q=\frac{5}{3}}d\tau\\
&&+(1+t)^{-2}\Big(1+\frac{(|x|-{\rm c}_1t)^2}{1+t}\Big)^{-N}\int_{0}^{t}(1+\tau)^{-\frac{3}{2}}d\tau\\
&\lesssim&\!\! (1+t)^{-2}\Big(1+\frac{(|x|-{\rm c}_1t)^2}{1+t}\Big)^{-1}.
\end{eqnarray*}

\vspace{1.5mm}
\textbf{Case 2.2:} $\frac{{\rm c}_1t-|x|}{2({\rm c}_1+{\rm c}_2)}\leq \tau\leq\frac{t}{2}$. We have
\begin{eqnarray*}
\arraycolsep=1.5pt
J_2 & \triangleq &\int_{\frac{{\rm c}_1t-|x|}{2({\rm c}_1+{\rm c}_2)}}^{\frac{t}{2}}\left( \int_{\big|\left \vert y\right \vert -{\rm c}_2%
\tau\big| \leq \sqrt{1+t}}+\int_{\big|\left \vert y\right \vert -{\rm c}_2%
\tau\big| \geq \sqrt{1+t}}\right) \left( \cdots \right) dyd\tau\\
&\lesssim& (1+t)^{-2}\int_{\frac{{\rm c}_1t-|x|}{2({\rm c}_1+{\rm c}_2)}}^{\frac{t}{2}}\int_{\big|\left \vert y\right \vert -{\rm c}_2%
\tau\big| \leq \sqrt{1+t}}\bigg(1+\frac{(|x-y|-{\rm c}_1(t-\tau))^2}{1+t-\tau}\bigg)^{-N}\bigg(1+\frac{(|y|-{\rm c}_2\tau)^2}{1+\tau}\bigg)^{-2}(1+\tau)^{-4}dyd\tau  \\
&&+(1+t)^{-2}\int_{\frac{{\rm c}_1t-|x|}{2({\rm c}_1+{\rm c}_2)}}^{\frac{t}{2}}\int_{\big|\left \vert y\right \vert -{\rm c}_2%
\tau\big| \geq \sqrt{1+t}}  \bigg(1+\frac{(|x-y|-{\rm c}_1(t-\tau))^2}{1+t-\tau}\bigg)^{-N}(1+\tau)^{-4}\\
&&\!\!\ \ \ \ \ \ \  \ \ \ \ \ \ \ \ \ \ \ \cdot\bigg(1+\frac{(|y|-{\rm c}_2\tau)^2}{1+t}\bigg)^{-2}\Big(\frac{1+\tau}{1+t}\Big)^{2} dyd\tau   \\
&\lesssim& (1+t)^{-2}\bigg(1+\frac{{\rm c}_1t-|x|}{2({\rm c}_1+{\rm c}_2)}\bigg)^{-4}(1+t)^3+(1+t)^{-2}(1+t)^{-2}\bigg(1+\frac{{\rm c}_1t-|x|}{2({\rm c}_1+{\rm c}_2)}\bigg)^{-2}(1+t)^3\\
&\lesssim& (1+t)^{-2}\bigg(1+\frac{({\rm c}_1t-|x|)^2}{1+t}\bigg)^{-1}(1+t)^{-1}\bigg(1+\frac{({\rm c}_1t)^2}{4({\rm c}_1+{\rm c}_2)}\bigg)^{-1}(1+t)^3\\
&&+(1+t)^{-2}(1+t)^{-2}\bigg(1+\frac{({\rm c}_1t-|x|)^2}{1+t}\bigg)^{-1}(1+t)^{-1}(1+t)^3\\
&\lesssim&(1+t)^{-2}\bigg(1+\frac{({\rm c}_1t-|x|)^2}{1+t}\bigg)^{-1},
\end{eqnarray*}
where Lemma \ref{A.7}, Lemma \ref{A.8} and the fact ${\rm c}_1t-|x|\geq\frac{{\rm c}_1t}{2}$ are used.

\vspace{1.5mm}
\textbf{Case 2.3:} $\frac{t}{2}\leq \tau\leq t-\frac{|x|}{4{\rm c}_2}$. Then $\frac{|x|}{4{\rm c}_2}\leq t-\tau\leq\frac{t}{2}$. By using Lemma \ref{A.8}, one has
\begin{eqnarray*}
J_3 &\triangleq&\int_{\frac{t}{2}}^{t-\frac{|x|}{4{\rm c}_2}}\!\int_{\mathbb{R}^3}(1+t-\tau)^{-2}\Big(1+\frac{(|x-y|-{\rm c}_1(t-\tau))^2}{1+t-\tau}\Big)^{-N}(1+\tau)^{-4}\Big(1+\frac{(|y|-{\rm c}_2\tau)^2}{1+\tau}\Big)^{-2}dyd\tau\\
&\lesssim& \bigg(1+\frac{|x|}{4{\rm c}_2}\bigg)^{-4}(1+t)^{-4}(1+t)^3\lesssim (1+t)^{-\frac{3}{2}}\bigg(1+\frac{|x|^2}{1+t}\bigg)^{-\frac{3}{2}}\\
&&\ \ \ \ \ \ \ \ \  \ \ \ \ \ \ \ \ \  \ \ \ \ \ \ \ \ \ \ \ \ \ \ \ {\rm or},\ \ \ \ \lesssim(1+t)^{-2}\bigg(1+\frac{|x|^2}{1+t}\bigg)^{-1}.
\end{eqnarray*}

\vspace{1.5mm}
\textbf{Case 2.4:} $t-\frac{|x|}{4{\rm c}_2}\leq \tau\leq t$. At first, one has that when $|x-y|\leq\frac{|x|}{2}$,
\begin{eqnarray*}
{\rm c}_2\tau-|y|\geq {\rm c}_2t-\frac{|x|}{4}-|x|-|x-y|\geq {\rm c}_1t-\frac{|x|}{4}-|x|-|x-y|\geq\frac{|x|}{4},
\end{eqnarray*}
and when $|x-y|\geq\frac{|x|}{2}$,
\begin{eqnarray*}
|x-y|-{\rm c}_1(t-\tau)\geq\frac{|x|}{2}-{\rm c}_1t+{\rm c}_1t-\frac{{\rm c}_1|x|}{4{\rm c}_2}\geq\frac{|x|}{4}.
\end{eqnarray*}
Then, by using Lemma \ref{A.7} again, one has
\begin{eqnarray*}
J_4 &\triangleq&\int_{t-\frac{|x|}{4{\rm c}_2}}^t\!\big(\int_{|x-y|\leq\frac{|x|}{2}}+\int_{|x-y|\geq\frac{|x|}{2}}\big)(\cdots)dyd\tau\\
&\lesssim& (1+t)^{-4}\int_{t-\frac{|x|}{4{\rm c}_2}}^t\int_{|x-y|\leq\frac{|x|}{2}}(1+t-\tau)^{-2}\Big(1+\frac{(|x-y|-{\rm c}_1(t-\tau))^2}{1+t-\tau}\Big)^{-N}\bigg(1+\frac{|x|^2}{1+t}\bigg)^{-2}\bigg(\frac{1+\tau}{1+t}\bigg)^2dyd\tau\\
&&+(1+t)^{-4}\int_{t-\frac{|x|}{4{\rm c}_2}}^t\int_{|x-y|\geq\frac{|x|}{2}}(1+t-\tau)^{-2}\bigg(1+\frac{|x|^2}{1+t}\bigg)^{-N}\bigg(\frac{1+t-\tau}{1+t}\bigg)^N\Big(1+\frac{(|y|-{\rm c}_2\tau)^2}{1+\tau}\Big)^{-2}dyd\tau\\
&\lesssim& (1+t)^{-4}\bigg(1+\frac{|x|^2}{1+t}\bigg)^{-2}\int_{t-\frac{|x|}{4{\rm c}_2}}^t(1+t-\tau)^{-2}(1+t-\tau)^{\frac{5}{2}}\bigg(\frac{1+\tau}{1+t}\bigg)^2d\tau\\
&&+(1+t)^{-4}\bigg(1+\frac{|x|^2}{1+t}\bigg)^{-N}\int_{t-\frac{|x|}{4{\rm c}_2}}^t(1+t-\tau)^{-2}\bigg(\frac{1+t-\tau}{1+t}\bigg)^N(1+\tau)^{\frac{5}{2}}d\tau\\
&\lesssim& (1+t)^{-\frac{5}{2}}\bigg(1+\frac{|x|^2}{1+t}\bigg)^{-2}.
\end{eqnarray*}

\vspace{1.8mm}
\noindent\textbf{Case 3:} $\left( x,t\right) \in D_{22}:=\left \{ \frac{{\rm c}_1t}{2}\leq|x|\leq {\rm c}_1t-\sqrt{1+t}\right \}$. We study this case by dividing the time interval into five parts.

\vspace{1.8mm}
\textbf{Case 3.1:} $0\leq\tau\leq \frac{{\rm c}_1t-|x|}{2({\rm c}_1+{\rm c}_2)}$. Similar to the Case 2.1, one has
\begin{eqnarray*}
J_5&\triangleq&\int_0^{\frac{{\rm c}_1t-|x|}{2({\rm c}_1+{\rm c}_2)}}\int_{\mathbb{R}^3}(1+t-\tau)^{-2}\Big(1+\frac{(|x-y|-{\rm c}_1(t-\tau))^2}{1+t-\tau}\Big)^{-N}(1+\tau)^{-4}\Big(1+\frac{(|y|-{\rm c}_2\tau)^2}{1+\tau}\Big)^{-2}dyd\tau\\
&\lesssim& (1+t)^{-2}\Big(1+\frac{(|x|-{\rm c}_1t)^2}{1+t}\Big)^{-1}.
\end{eqnarray*}

\vspace{1.8mm}
\textbf{Case 3.2:} $\frac{{\rm c}_1t-|x|}{2({\rm c}_1+{\rm c}_2)}\leq\tau\leq \frac{t}{2}$. We use the spherical coordinates to obtain
\begin{eqnarray*}
J_6&\triangleq&\int_{\frac{{\rm c}_1t-|x|}{2({\rm c}_1+{\rm c}_2)}}^{\frac{t}{2}}\int_{\mathbb{R}^3}(1+t-\tau)^{-2}\Big(1+\frac{(|x-y|-{\rm c}_1(t-\tau))^2}{1+t-\tau}\Big)^{-N}(1+\tau)^{-4}\Big(1+\frac{(|y|-{\rm c}_2\tau)^2}{1+\tau}\Big)^{-2}dyd\tau\\
 &\lesssim &\int_{\frac{{\rm c}_1t-|x|}{2({\rm c}_1+{\rm c}_2)}}^{\frac{t}{2}}\int_{0}^{\infty }\int_{0}^{\pi }\left( 1+t-\tau
\right) ^{-2}\bigg(1+\frac{( \sqrt{\left \vert x\right \vert
^{2}+r^{2}-2r\left \vert x\right \vert \cos \theta }-{\rm c}_1\left( t-\tau
\right) ) ^{2}}{ 1+t-\tau}\bigg)^{-N}\\
&&\cdot\left( 1+\tau \right)
^{-4}\bigg( 1+\frac{\left( r-{\rm c}_2\tau \right) ^{2}}{1+\tau }\bigg)
^{-2}r^{2}\sin \theta d\theta drd\tau \\
&\lesssim &\int_{\frac{{\rm c}_1t-|x|}{2({\rm c}_1+{\rm c}_2)}
}^{\frac{t}{2}}\int_{0}^{\infty }\int_{\left \vert \left \vert x\right \vert
-r\right \vert }^{\left \vert x\right \vert +r}\left( 1+t-\tau \right)
^{-2}\bigg(1+\frac{( z-{\rm c}_1\left( t-\tau
\right) ) ^{2}}{ 1+t-\tau}\bigg)^{-N}\\
&&\cdot\left( 1+\tau \right) ^{-4}\bigg( 1+\frac{%
\left( r-{\rm c}_2\tau \right) ^{2}}{1+\tau }\bigg) ^{-2}rz\frac{1}{\left
\vert x\right \vert }dzdrd\tau \\
&\lesssim &\int_{\frac{{\rm c}_1t-|x|}{2({\rm c}_1+{\rm c}_2)}
}^{\frac{t}{2}}\int_{0}^{\infty }\int_{0}^{\infty }\left( 1+t-\tau \right)
^{-2}\bigg(1+\frac{( z-{\rm c}_1\left( t-\tau
\right) ) ^{2}}{ 1+t-\tau}\bigg)^{-N}\\
&&\cdot\left( 1+\tau \right) ^{-4}\bigg( 1+\frac{%
\left( r-{\rm c}_2\tau \right) ^{2}}{1+\tau }\bigg) ^{-2}rz\frac{1}{\left
\vert x\right \vert }dzdrd\tau \\
&\lesssim &\int_{\frac{{\rm c}_1t-|x|}{2({\rm c}_1+{\rm c}_2)}
}^{\frac{t}{2}}\int_{0}^{\infty }\left( 1+t-\tau \right) ^{-2+%
\frac{3}{2}}\left( 1+\tau \right) ^{-4}\bigg( 1+\frac{\left( r-{\rm c}_2%
\tau \right) ^{2}}{1+\tau }\bigg) ^{-2}\frac{r}{\left \vert x\right \vert }%
drd\tau \\
&\lesssim &\left( 1+t\right) ^{-2+\frac{3}{2}}\left \vert x\right
\vert ^{-1}\int_{\frac{{\rm c}_1t-|x|}{2({\rm c}_1+{\rm c}_2)}
}^{\frac{t}{2}}\left( 1+\tau \right) ^{-4+\frac{3}{2}}d\tau \\
&\lesssim &\left( 1+t\right) ^{-\frac{3}{2}}\left( 1+{\rm c}_1t-\left \vert x\right
\vert \right) ^{-\frac{3}{2}}\lesssim \left( 1+t\right) ^{-\frac{3}{2}}\left( 1+%
{\rm c}_1t-\left \vert x\right \vert \right) ^{\frac{1}{2}}\left( 1+{\rm
c}_1t-\left \vert x\right \vert \right) ^{-2} \\
&\lesssim &\left( 1+t\right)^{-2}\bigg( 1+\frac{\left( {\rm c}_1t-\left
\vert x\right \vert \right) ^{2}}{1+t}\bigg) ^{-1}\hbox{,}
\end{eqnarray*}%
where $z=\sqrt{\left \vert x\right \vert
^{2}+r^{2}-2r\left \vert x\right \vert \cos \theta }$ and we have used the fact ${\rm c}_1t-|x|\geq\sqrt{1+t}$ in $D_{12}$.

\vspace{1.8mm}
\textbf{Case 3.3:} $\frac{t}{2}\leq\tau\leq \frac{3|x|}{4{\rm c}_1}+\frac{t}{4}$. 
Then we use the spherical coordinates again to obtain%
\begin{eqnarray*}
J_7&\triangleq&\int_{\frac{t}{2}}^{\frac{3|x|}{4{\rm c}_1}+\frac{t}{4}}\int_{\mathbb{R}^3}(1+t-\tau)^{-2}\Big(1+\frac{(|x-y|-{\rm c}_1(t-\tau))^2}{1+t-\tau}\Big)^{-N}(1+\tau)^{-4}\Big(1+\frac{(|y|-{\rm c}_2\tau)^2}{1+\tau}\Big)^{-2}dyd\tau\\
&\lesssim &\int_{\frac{t}{2}}^{\frac{3|x|}{4{\rm c}_1}+\frac{t}{4}%
}\int_{0}^{\infty }\int_{0}^{\pi }\left( 1+t-\tau \right) ^{-2}\bigg(1+\frac{( \sqrt{\left \vert x\right \vert
^{2}+r^{2}-2r\left \vert x\right \vert \cos \theta }-{\rm c}_1\left( t-\tau
\right) ) ^{2}}{ 1+t-\tau}\bigg)^{-N}\\
&&\cdot\left( 1+t\right) ^{-4}\bigg( 1+\frac{\left(
r-{\rm c}_2\tau \right) ^{2}}{1+\tau }\bigg) ^{-2}r^{2}\sin \theta d\theta
drd\tau \\
&\lesssim &\int_{\frac{t}{2}}^{\frac{3|x|}{4{\rm c}_1}+\frac{t}{4}%
}\int_{0}^{\infty }\int_{\left \vert \left \vert x\right \vert -r\right
\vert }^{\left \vert x\right \vert +r}\left( 1+t-\tau \right) ^{-2}\bigg(1+\frac{( z-{\rm c}_1\left( t-\tau
\right) ) ^{2}}{ 1+t-\tau}\bigg)^{-N}\\
&&\cdot\left( 1+t\right) ^{-4}\bigg( 1+\frac{\left( r-{\rm c}_2%
\tau \right) ^{2}}{1+\tau }\bigg) ^{-2}rz\frac{1}{\left \vert x\right \vert
}dzdrd\tau \\
&\lesssim &\left( 1+t\right) ^{-4}\left \vert x\right \vert ^{-1}\int_{\frac{%
t}{2}}^{\frac{3|x|}{4{\rm c}_1}+\frac{t}{4}} \int_{0}^{\infty }\left(
1+t-\tau \right) ^{-2+\frac{3}{2}}\bigg( 1+\frac{\left( r-{\rm c}_2\tau \right) ^{2}}{1+\tau }\bigg) ^{-2}rdrd\tau \\
&\lesssim &\left( 1+t\right) ^{-\frac{7}{2}}\int_{\frac{t}{2}}^{\frac{3|x|}{4{\rm c}_1}+\frac{t}{4}}\left( 1+t-\tau \right) ^{-\frac{1}{2}}d\tau \\
&\lesssim &\left( 1+t\right) ^{-\frac{7}{2}}\bigg[\bigg(1+\frac{t}{2}\bigg)^{\frac{1}{2}}-
\bigg(1+\frac{{\rm c}_1t-|x|}{4{\rm c}_1}\bigg)^{\frac{1}{2}}\bigg] \\
&\lesssim& \left( 1+t\right) ^{-3}\lesssim\left( 1+t\right) ^{-2}\bigg( 1+\frac{\left( {\rm c}_1t-\left
\vert x\right \vert \right) ^{2}}{1+t}\bigg) ^{-1}\hbox{.}
\end{eqnarray*}

\vspace{1.8mm}
\textbf{Case 3.4:} $\frac{3|x|}{4{\rm c}_1}+\frac{t}{4}\leq\tau\leq \frac{|x|}{4{\rm c}_1}+\frac{3t}{4}$. Similar to $J_{7}$, we can directly have
\begin{eqnarray*}
J_{8}&\triangleq&\int_{\frac{3|x|}{4{\rm c}_1}+\frac{t}{4}}^{\frac{|x|}{4{\rm c}_1}+\frac{3t}{4}}\int_{\mathbb{R}^3}(1+t-\tau)^{-2}\Big(1+\frac{(|x-y|-{\rm c}_1(t-\tau))^2}{1+t-\tau}\Big)^{-N}(1+\tau)^{-4}\Big(1+\frac{(|y|-{\rm c}_2\tau)^2}{1+\tau}\Big)^{-2}dyd\tau\\
&\lesssim &\left( 1+t\right) ^{-\frac{7}{2}}\int_{\frac{3|x|}{4{\rm c}_1}+\frac{t}{4}}^{\frac{|x|}{4{\rm c}_1}+\frac{3t}{4}}\left( 1+t-\tau \right)
^{-\frac{1}{2}}d\tau \\
&\lesssim &\left( 1+t\right) ^{-\frac{7}{2}}\bigg[\left(1+\frac{3({\rm c}_1t-|x|)}{4%
{\rm c}_1}\right)^{\frac{1}{2}} -\left(1+\frac{({\rm c}_1t-|x|)}{4{\rm c}_1}\right)^{\frac{1}{2}}%
\bigg]\lesssim (1+t)^{-3}\\
&\lesssim& \left( 1+t\right) ^{-2}\bigg( 1+\frac{\left( {\rm c}_1t-\left
\vert x\right \vert \right) ^{2}}{1+t}\bigg) ^{-1}\hbox{.}
\end{eqnarray*}

\vspace{1.8mm}
\textbf{Case 3.5:} $\frac{|x|}{4{\rm c}_1}+\frac{3t}{4}\leq\tau\leq t$. We decompose $\mathbb{R}^{3}$ into two parts%
\begin{equation*}
\int_{\frac{|x|}{4{\rm c}_1}+\frac{3t}{4}}^{t}\left( \int_{\left \vert y\right \vert \leq \frac{%
\left \vert x\right \vert +{\rm c}_1t}{2}}+\int_{\left \vert y\right \vert >%
\frac{\left \vert x\right \vert +{\rm c}_1t}{2}}\right) \left( \cdots
\right) dyd\tau =:J_{9}+J_{10}\hbox{.}
\end{equation*}
If $\frac{|x|}{4{\rm c}_1}+\frac{3t}{4} \leq \tau \leq t$ and $\left \vert y\right \vert \leq \frac{%
\left \vert x\right \vert +{\rm c}_1t}{2}$, then%
\begin{equation*}
{\rm c}_2\tau -\left \vert y\right \vert \geq \frac{3{\rm c}_1t}{4}+\frac{\left \vert x\right \vert%
}{4}-\frac{\left \vert
x\right \vert +{\rm c}_1t}{2}= \frac{{\rm c}_1t-\left \vert x\right
\vert }{4}\hbox{.}
\end{equation*}%
If $\frac{|x|}{4{\rm c}_1}+\frac{3t}{4} \leq \tau \leq t$ and $\left \vert y\right \vert >\frac{%
\left \vert x\right \vert +{\rm c}_1t}{2}$, then%
\begin{eqnarray*}
\left \vert x-y\right \vert -{\rm c}_1\left( t-\tau \right) &\geq &\left
\vert y\right \vert -\left \vert x\right \vert -{\rm c}_1\left( t-\tau
\right) \geq \frac{\left \vert x\right \vert +{\rm c}_1t}{2}-\left \vert
x\right \vert -{\rm c}_1t+\frac{{\rm c}_1t}{2}+\frac{1}{4}\left( {\rm c}_1%
t+\left \vert x\right \vert \right) \geq \frac{{\rm c}_1t-\left \vert
x\right \vert }{4}\hbox{.}
\end{eqnarray*}%
Hence,
\begin{eqnarray*}
J_{9} &\lesssim &\int_{\frac{|x|}{4{\rm c}_1}+\frac{3t}{4}}^{t}\int_{\left \vert y\right \vert \leq
\frac{\left \vert x\right \vert +{\rm c}_1t}{2}}\left( 1+t-\tau \right)
^{-2}\Big(1+\frac{(|x-y|-{\rm c}_1(t-\tau))^2}{1+t-\tau}\Big)^{-N}\\
&&\cdot\left( 1+t\right)
^{-4}\bigg( 1+\frac{\left( {\rm c}_1t-\left \vert x\right \vert \right) ^{2}%
}{1+t}\bigg) ^{-1}\left( 1+\frac{\left( {\rm c}_2\tau-\left \vert y\right \vert \right) ^{2}%
}{1+t}\right) ^{-1}dyd\tau \\
&\lesssim &\left( 1+t\right) ^{-4}\bigg( 1+\frac{\left( {\rm c}_1t-\left
\vert x\right \vert \right) ^{2}}{1+t}\bigg) ^{-1}\int_{\frac{|x|}{4{\rm c}_1}+\frac{3t}{4}}^{t}(1+t-\tau)^{-2}(1+t-\tau)^{\frac{5}{4}}(1+\tau)^{\frac{3}{4}}d\tau\\
&\lesssim &\left( 1+t\right) ^{-3}\bigg( 1+\frac{\left( {\rm c}_1t-\left
\vert x\right \vert \right) ^{2}}{1+t}\bigg) ^{-1}\hbox{,}
\end{eqnarray*}%
\begin{eqnarray*}
J_{10} &\lesssim &\int_{\frac{|x|}{4{\rm c}_1}+\frac{3t}{4}}^{t}\int_{\left \vert y\right \vert >%
\frac{\left \vert x\right \vert +{\rm c}_1t}{2}}\left( 1+t-\tau \right)
^{-2}\Big(1+\frac{({\rm c}_1t-|x|)^2}{1+t}\Big)^{-N/2}\\
&&\cdot\Big(1+\frac{(|x-y|-{\rm c}_1(t-\tau))^2}{1+t-\tau}\Big)^{-N/2}\left( 1+t\right) ^{-4}\Big(1+\frac{(|y|-{\rm c}_2\tau)^2}{1+\tau}\Big)^{-2}dyd\tau \\
&\lesssim &\left( 1+t\right) ^{-4}\Big(1+\frac{({\rm c}_1t-|x|)^2}{1+t}\Big)^{-N/2}\int_{\frac{|x|}{4{\rm c}_1}+\frac{3t}{4}}^{t}\left( 1+t-\tau \right)
^{-2+\frac{5}{2}}d\tau \\
&\lesssim &\left( 1+t\right) ^{-\frac{5}{2}}\Big(1+\frac{({\rm c}_1t-|x|)^2}{1+t}\Big)^{-N/2}\hbox{.}
\end{eqnarray*}%

\vspace{1.8mm}
\noindent\textbf{Case 4:} $\left( x,t\right) \in D_{32}:=\left \{ \frac{({\rm c}_1+{\rm c}_2)t}{2}\leq|x|\leq {\rm c}_2t-\sqrt{1+t}\right \}$. It's easy to see that $(|x|-{\rm c}_1t)^2\geq 1+t$ and $({\rm c}_2t-|x|)^2\geq 1+t$.

\textbf{Case 4.1:} $0\leq\tau\leq \frac{|x|-{\rm c}_1t}{2({\rm c}_2-{\rm c}_1)}$. Obviously, $\tau\leq\frac{t}{2}$. We decompose $\mathbb{R}^{3}$ into two parts%
\begin{equation*}
J_{11}\triangleq\int_0^{\frac{|x|-{\rm c}_1t}{2({\rm c}_2-{\rm c}_1)}}\left( \int_{\vert y\vert \geq \frac{|x|-{\rm c}_1t+({\rm c}_1+{\rm c}_2)\tau}{2}}+\int_{\vert y\vert \leq \frac{|x|-{\rm c}_1t+({\rm c}_1+{\rm c}_2)\tau}{2}}\right) ( \cdots) dyd\tau \hbox{.}
\end{equation*}
If $\vert y\vert \geq \frac{|x|-{\rm c}_1t+({\rm c}_1+{\rm c}_2)\tau}{2}$, then%
\begin{equation*}
\left \vert y\right \vert-{\rm c}_2\tau \geq \frac{|x|-{\rm c}_1t}{2}-\frac{{\rm c}_2-{\rm c}_1}{2}\tau\geq\frac{|x|-{\rm c}_1t}{4}.
\end{equation*}%
If $\vert y\vert \leq \frac{|x|-{\rm c}_1t+({\rm c}_1+{\rm c}_2)\tau}{2}$, then%
\begin{eqnarray*}
\left \vert x-y\right \vert -{\rm c}_1\left( t-\tau \right) \geq |x|-{\rm c}_1t+{\rm c}_1\tau-\frac{|x|-{\rm c}_1t+({\rm c}_1+{\rm c}_2)\tau}{2}\geq\frac{|x|-{\rm c}_1t}{4}.
\end{eqnarray*}%
Hence, similar to Case 2.1, one has
\begin{eqnarray*}
J_{11} \lesssim (1+t)^{-2}\Big(1+\frac{(|x|-{\rm c}_1t)^2}{1+t}\Big)^{-1}.
\end{eqnarray*}%

\textbf{Case 4.2:} $\frac{|x|-{\rm c}_1t}{2({\rm c}_2-{\rm c}_1)}\leq\tau\leq \frac{t}{2}$. Similar to Case 2.2, we have
\begin{eqnarray*}
&&J_{12}\triangleq\int_{\frac{|x|-{\rm c}_1t}{2({\rm c}_2-{\rm c}_1)}}^{\frac{t}{2}}\left( \int_{\big|\left \vert y\right \vert -{\rm c}_2%
\tau\big| \leq \sqrt{1+t}}+\int_{\big|\left \vert y\right \vert -{\rm c}_2%
\tau\big| \geq \sqrt{1+t}}\right) \left( \cdots \right) dyd\tau\\
&\lesssim& (1+t)^{-2}\int_{\frac{|x|-{\rm c}_1t}{2({\rm c}_2-{\rm c}_1)}}^{\frac{t}{2}}\int_{\big|\left \vert y\right \vert -{\rm c}_2%
\tau\big| \leq \sqrt{1+t}}\bigg(1+\frac{(|x-y|-{\rm c}_1(t-\tau))^2}{1+t-\tau}\bigg)^{-N}\bigg(1+\frac{(|y|-{\rm c}_2\tau)^2}{1+\tau}\bigg)^{-2}(1+\tau)^{-4}dyd\tau  \\
&&+(1+t)^{-2}\int_{\frac{|x|-{\rm c}_1t}{2({\rm c}_2-{\rm c}_1)}}^{\frac{t}{2}}\int_{\big|\left \vert y\right \vert -{\rm c}_2%
\tau\big| \geq \sqrt{1+t}}  \bigg(1+\frac{(|x-y|-{\rm c}_1(t-\tau))^2}{1+t-\tau}\bigg)^{-N}(1+\tau)^{-2}\\
&&\!\!\ \ \ \ \ \ \  \ \ \ \ \ \ \ \ \ \ \ \cdot\bigg(1+\frac{(|y|-{\rm c}_2\tau)^2}{1+t}\bigg)^{-2}(1+t)^{-2} dyd\tau   \\
&\lesssim& (1+t)^{-2}\bigg(1+\frac{|x|-{\rm c}_1t}{2({\rm c}_1+{\rm c}_2)}\bigg)^{-4}(1+t)^3+(1+t)^{-2}(1+t)^{-2}\bigg(1+\frac{|x|-{\rm c}_1t}{2({\rm c}_1+{\rm c}_2)}\bigg)^{-2}(1+t)^3\\
&\lesssim& (1+t)^{-2}\bigg(1+\frac{(|x|-{\rm c}_1t)^2}{1+t}\bigg)^{-1}(1+t)^{-1}\bigg(1+\frac{(({\rm c}_2-{\rm c}_1)t)^2}{4({\rm c}_1+{\rm c}_2)}\bigg)^{-1}(1+t)^3\\
&&+(1+t)^{-2}(1+t)^{-2}\bigg(1+\frac{(|x|-{\rm c}_1t)^2}{1+t}\bigg)^{-1}(1+t)^{-1}(1+t)^3\\
&\lesssim&(1+t)^{-2}\bigg(1+\frac{(|x|-{\rm c}_1t)^2}{1+t}\bigg)^{-1}.
\end{eqnarray*}

\textbf{Case 4.3:} $\frac{t}{2} \leq\tau\leq t-\frac{{\rm c}_2t-|x|}{4{\rm c}_2}$. We know that $\frac{{\rm c}_2t-|x|}{4{\rm c}_2}\leq t-\tau\leq \frac{t}{2}$. Then by using Lemma \ref{A.8}, one has
\begin{eqnarray*}
J_{13}&\triangleq&\int_{\frac{t}{2}}^{t-\frac{{\rm c}_2t-|x|}{4{\rm c}_2}}\int_{\mathbb{R}^3}(1+t-\tau)^{-2}\Big(1+\frac{(|x-y|-{\rm c}_1(t-\tau))^2}{1+t-\tau}\Big)^{-N}(1+\tau)^{-4}\Big(1+\frac{(|y|-{\rm c}_2\tau)^2}{1+\tau}\Big)^{-2}dyd\tau\\
&\lesssim &\bigg(1+\frac{{\rm c}_2t-|x|}{4{\rm c}_2}\bigg)^{-2}(1+t)^{-4}\int_{0}^{t}\int_{\mathbb{R}^3}\Big(1+\frac{(|x-y|-{\rm c}_1(t-\tau))^2}{1+t-\tau}\Big)^{-N}\Big(1+\frac{(|y|-{\rm c}_2\tau)^2}{1+\tau}\Big)^{-2}dyd\tau \\
&\lesssim &\bigg(1+\frac{{\rm c}_2t-|x|}{4{\rm c}_2}\bigg)^{-2}(1+t)^{-4}(1+t)^3\\
&\lesssim&(1+t)^{-2}\bigg(1+\frac{({\rm c}_2t-|x|)^2}{1+t}\bigg)^{-1}.
\end{eqnarray*}

\textbf{Case 4.4:} $t-\frac{{\rm c}_2t-|x|}{4{\rm c}_2} \leq\tau\leq t$. We decompose $\mathbb{R}^{3}$ into two parts%
\begin{equation*}
\int_{t-\frac{{\rm c}_2t-|x|}{4{\rm c}_2}}^{t}\left( \int_{\left \vert x-y\right \vert \leq \frac{{\rm c}_2t-|x|}{2}}+\int_{\left \vert x-y\right \vert \geq \frac{{\rm c}_2t-|x|}{2}}\right) \left( \cdots
\right) dyd\tau =:J_{14}+J_{15}\hbox{.}
\end{equation*}
If $\left \vert x-y\right \vert \leq \frac{{\rm c}_2t-|x|}{2}$, then%
\begin{equation*}
{\rm c}_2\tau -\left \vert y\right \vert \geq {\rm c}_2t-\frac{{\rm c}_2t-|x|}{4}-|x|-|x-y|\geq \frac{{\rm c}_2t-|x|}{4}.
\end{equation*}%
If $\left \vert x-y\right \vert \geq \frac{{\rm c}_2t-|x|}{2}$, then%
\begin{eqnarray*}
\left \vert x-y\right \vert -{\rm c}_1\left( t-\tau \right) \geq \frac{{\rm c}_2t-|x|}{2}-\frac{{\rm c}_1({\rm c}_2t-|x|)}{4{\rm c}_2}\geq \frac{{\rm c}_2t-|x|}{4}.
\end{eqnarray*}%
Hence,
\begin{eqnarray*}
J_{14} &\lesssim &\int_{t-\frac{{\rm c}_2t-|x|}{4{\rm c}_2}}^{t}\int_{\left \vert x-y\right \vert \leq \frac{{\rm c}_2t-|x|}{2}}\left( 1+t-\tau \right)
^{-2}\Big(1+\frac{(|x-y|-{\rm c}_1(t-\tau))^2}{1+t-\tau}\Big)^{-N}\\
&&\cdot\left( 1+t\right)
^{-4}\bigg( 1+\frac{\left( {\rm c}_2t-\left \vert x\right \vert \right) ^{2}%
}{1+t}\bigg) ^{-1}\left( 1+\frac{\left( {\rm c}_2\tau-\left \vert y\right \vert \right) ^{2}%
}{1+t}\right) ^{-1}dyd\tau \\
&\lesssim &\left( 1+t\right) ^{-4}\bigg( 1+\frac{\left( {\rm c}_2t-\left
\vert x\right \vert \right) ^{2}}{1+t}\bigg) ^{-1}\int_{t-\frac{{\rm c}_2t-|x|}{4{\rm c}_2}}^{t}(1+t-\tau)^{-2}(1+t-\tau)^{\frac{5}{4}}(1+\tau)^{\frac{3}{4}}d\tau\\
&\lesssim &\left( 1+t\right) ^{-3}\bigg( 1+\frac{\left( {\rm c}_2t-\left
\vert x\right \vert \right) ^{2}}{1+t}\bigg) ^{-1}\hbox{,}
\end{eqnarray*}%
\begin{eqnarray*}
J_{15} &\lesssim &\int_{t-\frac{{\rm c}_2t-|x|}{4{\rm c}_2}}^{t}\int_{\left \vert x-y\right \vert \geq \frac{{\rm c}_2t-|x|}{2}}\left( 1+t-\tau \right)
^{-2}\Big(1+\frac{({\rm c}_2t-|x|)^2}{1+t}\Big)^{-N/2}\\
&&\cdot\Big(1+\frac{(|x-y|-{\rm c}_1(t-\tau))^2}{1+t-\tau}\Big)^{-N/2}\left( 1+t\right) ^{-4}\Big(1+\frac{(|y|-{\rm c}_2\tau)^2}{1+\tau}\Big)^{-2}dyd\tau \\
&\lesssim &\left( 1+t\right) ^{-4}\Big(1+\frac{({\rm c}_2t-|x|)^2}{1+t}\Big)^{-N/2}\int_{t-\frac{{\rm c}_2t-|x|}{4{\rm c}_2}}^{t}\left( 1+t-\tau \right)
^{-2+\frac{5}{2}}d\tau \\
&\lesssim &\left( 1+t\right) ^{-\frac{5}{2}}\Big(1+\frac{({\rm c}_2t-|x|)^2}{1+t}\Big)^{-N/2}\hbox{.}
\end{eqnarray*}

\noindent\textbf{Case 5:} $\left( x,t\right) \in D_{33}:=\bigg\{{\rm c}_1t+\sqrt{1+t}\leq |x|\leq  \frac{({\rm c}_1+{\rm c}_2)t}{2} \bigg\}$. We study this case by dividing the time interval into four parts.

\textbf{Case 5.1:} $0\leq\tau\leq \frac{|x|-{\rm c}_1t}{2({\rm c}_2-{\rm c}_1)}$. Obviously, $\tau\leq\frac{t}{2}$. We decompose $\mathbb{R}^{3}$ into two parts%
\begin{equation*}
J_{16}\triangleq\int_0^{\frac{|x|-{\rm c}_1t}{2({\rm c}_2-{\rm c}_1)}}\left( \int_{\vert y\vert \geq \frac{|x|-{\rm c}_1t+({\rm c}_1+{\rm c}_2)\tau}{2}}+\int_{\vert y\vert \leq \frac{|x|-{\rm c}_1t+({\rm c}_1+{\rm c}_2)\tau}{2}}\right) ( \cdots) dyd\tau \hbox{.}
\end{equation*}
If $\vert y\vert \geq \frac{|x|-{\rm c}_1t+({\rm c}_1+{\rm c}_2)\tau}{2}$, then%
\begin{equation*}
\left \vert y\right \vert-{\rm c}_2\tau \geq \frac{|x|-{\rm c}_1t}{2}-\frac{{\rm c}_2-{\rm c}_1}{2}\tau\geq\frac{|x|-{\rm c}_1t}{4}.
\end{equation*}%
If $\vert y\vert \leq \frac{|x|-{\rm c}_1t+({\rm c}_1+{\rm c}_2)\tau}{2}$, then%
\begin{eqnarray*}
\left \vert x-y\right \vert -{\rm c}_1\left( t-\tau \right) \geq |x|-{\rm c}_1t+{\rm c}_1\tau-\frac{|x|-{\rm c}_1t+({\rm c}_1+{\rm c}_2)\tau}{2}\geq\frac{|x|-{\rm c}_1t}{4}.
\end{eqnarray*}%
Hence, similar to Case 2.1, one has
\begin{eqnarray*}
\arraycolsep=1.5pt
J_{16}\lesssim (1+t)^{-2}\Big(1+\frac{(|x|-{\rm c}_1t)^2}{1+t}\Big)^{-1}.
\end{eqnarray*}

\textbf{Case 5.2:} $\frac{|x|-{\rm c}_1t}{2({\rm c}_2-{\rm c}_1)}\leq\tau\leq \frac{t}{2}$. By using  the spherical coordinates as in Case 3.2, one has
\begin{eqnarray*}
J_{17}&\triangleq&\int_{\frac{|x|-{\rm c}_1t}{2({\rm c}_2-{\rm c}_1)}}^{\frac{t}{2}}\int_{\mathbb{R}^3}(1+t-\tau)^{-2}\Big(1+\frac{(|x-y|-{\rm c}_1(t-\tau))^2}{1+t-\tau}\Big)^{-N}(1+\tau)^{-4}\Big(1+\frac{(|y|-{\rm c}_2\tau)^2}{1+\tau}\Big)^{-2}dyd\tau\\
 &\lesssim &\int_{\frac{|x|-{\rm c}_1t}{2({\rm c}_2-{\rm c}_1)}}^{\frac{t}{2}}\int_{0}^{\infty }\int_{0}^{\pi }\left( 1+t-\tau
\right) ^{-2}\bigg(1+\frac{( \sqrt{\left \vert x\right \vert
^{2}+r^{2}-2r\left \vert x\right \vert \cos \theta }-{\rm c}_1\left( t-\tau
\right) ) ^{2}}{ 1+t-\tau}\bigg)^{-N}\\
&&\ \ \ \ \ \ \ \ \ \ \ \ \ \ \ \ \ \cdot\left( 1+\tau \right)
^{-4}\bigg( 1+\frac{\left( r-{\rm c}_2\tau \right) ^{2}}{1+\tau }\bigg)
^{-2}r^{2}\sin \theta d\theta drd\tau \\
&\lesssim &\int_{\frac{|x|-{\rm c}_1t}{2({\rm c}_2-{\rm c}_1)}
}^{\frac{t}{2}}\int_{0}^{\infty }\int_{0}^{\infty }\left( 1+t-\tau \right)
^{-2}\bigg(1+\frac{( z-{\rm c}_1\left( t-\tau
\right) ) ^{2}}{ 1+t-\tau}\bigg)^{-N}\\
&&\ \ \ \ \ \ \ \ \ \ \ \ \ \ \ \ \ \cdot\left( 1+\tau \right) ^{-4}\bigg( 1+\frac{%
\left( r-{\rm c}_2\tau \right) ^{2}}{1+\tau }\bigg) ^{-2}rz\frac{1}{\left
\vert x\right \vert }dzdrd\tau \\
&\lesssim &\int_{\frac{|x|-{\rm c}_1t}{2({\rm c}_2-{\rm c}_1)}
}^{\frac{t}{2}}\int_{0}^{\infty }\left( 1+t-\tau \right) ^{-2+%
\frac{3}{2}}\left( 1+\tau \right) ^{-4}\bigg( 1+\frac{\left( r-{\rm c}_2%
\tau \right) ^{2}}{1+\tau }\bigg) ^{-2}\frac{r}{\left \vert x\right \vert }%
drd\tau \\
&\lesssim &\left( 1+t\right) ^{-2+\frac{3}{2}}\left \vert x\right
\vert ^{-1}\int_{\frac{|x|-{\rm c}_1t}{2({\rm c}_2-{\rm c}_1)}
}^{\frac{t}{2}}\left( 1+\tau \right) ^{-4+\frac{3}{2}}d\tau \\
&\lesssim &\left( 1+t\right)^{-2}\bigg( 1+\frac{\left( \left
\vert x\right \vert-{\rm c}_1t \right) ^{2}}{1+t}\bigg) ^{-1}\hbox{.}
\end{eqnarray*}

\textbf{Case 5.3:} $\frac{t}{2} \leq\tau\leq t-\frac{{\rm c}_2t-|x|}{4{\rm c}_2}$. It is easy to see that $\frac{{\rm c}_1t-|x|}{4{\rm c}_2}\leq t-\tau\leq\frac{t}{2}$. Then by using Lemma \ref{A.8}, one has
\begin{eqnarray*}
J_{18}&\triangleq&\int_{\frac{t}{2}}^{t-\frac{{\rm c}_2t-|x|}{4{\rm c}_2}}\int_{\mathbb{R}^3}(1+t-\tau)^{-2}\Big(1+\frac{(|x-y|-{\rm c}_1(t-\tau))^2}{1+t-\tau}\Big)^{-N}(1+\tau)^{-4}\Big(1+\frac{(|y|-{\rm c}_2\tau)^2}{1+\tau}\Big)^{-2}dyd\tau\\
&\lesssim &\bigg(1+\frac{{\rm c}_1t-|x|}{4{\rm c}_2}\bigg)^{-2}(1+t)^{-4}\int_{0}^{t}\int_{\mathbb{R}^3}\Big(1+\frac{(|x-y|-{\rm c}_1(t-\tau))^2}{1+t-\tau}\Big)^{-N}\Big(1+\frac{(|y|-{\rm c}_2\tau)^2}{1+\tau}\Big)^{-2}dyd\tau \\
&\lesssim &\bigg(1+\frac{{\rm c}_1t-|x|}{4{\rm c}_2}\bigg)^{-2}(1+t)^{-4}(1+t)^3\\
&\lesssim&(1+t)^{-2}\bigg(1+\frac{({\rm c}_1t-|x|)^2}{1+t}\bigg)^{-1}.
\end{eqnarray*}

\textbf{Case 5.4:} $t-\frac{{\rm c}_2t-|x|}{4{\rm c}_2} \leq\tau\leq t$. We decompose $\mathbb{R}^{3}$ into two parts:

if $\left \vert x-y\right \vert \leq \frac{{\rm c}_2t-|x|}{2}$, then%
\begin{equation*}
{\rm c}_2\tau -\left \vert y\right \vert \geq {\rm c}_2t-\frac{{\rm c}_2t-|x|}{4}-|x|-|x-y|\geq \frac{{\rm c}_2t-|x|}{4};
\end{equation*}

if $\left \vert x-y\right \vert \geq \frac{{\rm c}_2t-|x|}{2}$, then%
\begin{eqnarray*}
\left \vert x-y\right \vert -{\rm c}_1\left( t-\tau \right) \geq \frac{{\rm c}_2t-|x|}{2}-\frac{{\rm c}_1({\rm c}_2t-|x|)}{4{\rm c}_2}\geq \frac{{\rm c}_2t-|x|}{4}.
\end{eqnarray*}%
Hence,
\begin{eqnarray*}
J_{19} &\triangleq&\int_{t-\frac{{\rm c}_2t-|x|}{4{\rm c}_2}}^{t}\left( \int_{\left \vert x-y\right \vert \leq \frac{{\rm c}_2t-|x|}{2}}+\int_{\left \vert x-y\right \vert \geq \frac{{\rm c}_2t-|x|}{2}}\right) \left( \cdots
\right) dyd\tau\\
&\lesssim &\int_{t-\frac{{\rm c}_2t-|x|}{4{\rm c}_2}}^{t}\int_{\left \vert x-y\right \vert \leq \frac{{\rm c}_2t-|x|}{2}}\left( 1+t-\tau \right)
^{-2}\Big(1+\frac{(|x-y|-{\rm c}_1(t-\tau))^2}{1+t-\tau}\Big)^{-N}\\
&&\ \ \ \ \ \ \ \ \ \ \ \ \ \ \ \ \ \ \cdot\left( 1+t\right)
^{-4}\bigg( 1+\frac{\left( {\rm c}_2t-\left \vert x\right \vert \right) ^{2}%
}{1+t}\bigg) ^{-1}\left( 1+\frac{\left( {\rm c}_2\tau-\left \vert y\right \vert \right) ^{2}%
}{1+t}\right) ^{-1}dyd\tau \\
&&+\int_{t-\frac{{\rm c}_2t-|x|}{4{\rm c}_2}}^{t}\int_{\left \vert x-y\right \vert \geq \frac{{\rm c}_2t-|x|}{2}}\left( 1+t-\tau \right)
^{-2}\Big(1+\frac{({\rm c}_2t-|x|)^2}{1+t}\Big)^{-N/2}\\
&&\ \ \ \ \ \ \ \ \ \ \ \ \ \ \ \ \ \ \ \cdot\Big(1+\frac{(|x-y|-{\rm c}_1(t-\tau))^2}{1+t-\tau}\Big)^{-N/2}\left( 1+t\right) ^{-4}\Big(1+\frac{(|y|-{\rm c}_2\tau)^2}{1+\tau}\Big)^{-2}dyd\tau \\
&\lesssim &\left( 1+t\right) ^{-4}\bigg( 1+\frac{\left( {\rm c}_2t-\left
\vert x\right \vert \right) ^{2}}{1+t}\bigg) ^{-1}\int_{t-\frac{{\rm c}_2t-|x|}{4{\rm c}_2}}^{t}(1+t-\tau)^{-2}(1+t-\tau)^{\frac{5}{4}}(1+\tau)^{\frac{3}{4}}d\tau\\
&&+\left( 1+t\right) ^{-4}\Big(1+\frac{({\rm c}_2t-|x|)^2}{1+t}\Big)^{-N/2}\int_{t-\frac{{\rm c}_2t-|x|}{4{\rm c}_2}}^{t}\left( 1+t-\tau \right)
^{-2+\frac{5}{2}}d\tau \\
&\lesssim &\left( 1+t\right) ^{-\frac{5}{2}}\bigg( 1+\frac{\left( {\rm c}_2t-\left
\vert x\right \vert \right) ^{2}}{1+t}\bigg) ^{-1}\hbox{.}
\end{eqnarray*}%

\noindent\textbf{Case 6:} $\left( x,t\right) \in D_{33}:=\bigg\{|x|\geq  {\rm c}_2t+\sqrt{1+t} \bigg\}$. It's obvious that $(|x|-{\rm c}_2t)^2\geq 1+t$. We first consider two subcases. The first one is $\big||y|-{\rm c}_2\tau\big|\geq\frac{|x|-{\rm c}_2t}{2}\geq0$. The second one is $\big||y|-{\rm c}_2\tau\big|\leq\frac{|x|-{\rm c}_2t}{2}$, which implies
\begin{eqnarray*}
\left \vert x-y\right \vert -{\rm c}_1\left( t-\tau \right) \geq |x|-|y|-{\rm c}_2t+{\rm c}_2\tau\geq \frac{|x|-{\rm c}_2t}{2}.
\end{eqnarray*}
Then when $0\leq \tau\leq\frac{t}{2}$, by using Lemma \ref{A.7} again, one has


 \begin{eqnarray*}
J_{20} &\triangleq&\!\!\int_{0}^{\frac{t}{2}}\left( \int_{\big||y|-{\rm c}_2\tau\big|\geq\frac{|x|-{\rm c}_2t}{2}}+\int_{\big||y|-{\rm c}_2\tau\big|\leq\frac{|x|-{\rm c}_2t}{2}}\right) \left( \cdots \right) dyd\tau\\
&\lesssim\!\!&\!\! (1+t)^{-1}\Big(1+\frac{(|x|-{\rm c}_2t)^2}{1+t}\Big)^{-1}\!\int_{0}^{\frac{t}{2}}\!\!\int_{\big||y|-{\rm c}_2\tau\big|\geq\frac{|x|-{\rm c}_2t}{2}}(1+t-\tau)^{-1}\Big(1+\frac{(|x-y|-{\rm c}_1(t-\tau))^2}{1+t-\tau}\Big)^{-N}\\
&&\ \ \ \ \ \ \  \ \ \ \ \ \ \ \ \ \ \ \  \ \ \ \ \ \ \ \ \ \ \ \  \ \ \ \ \ \ \ \ \ \ \ \  \ \ \ \ \ \ \ \ \ \ \ \cdot(1+\tau)^{-4}\bigg(\frac{1+\tau}{1+t}\bigg)\Big(1+\frac{(|y|-{\rm c}_2\tau)^2}{1+\tau}\Big)^{-1}dyd\tau\\
&&\!\!\!+(1+t)^{-2}\Big(1+\frac{(|x|-{\rm c}_2t)^2}{1+t}\Big)^{-N}\!\!\int_{0}^{\frac{t}{2}}\int_{\big||y|-{\rm c}_2\tau\big|\leq\frac{|x|-{\rm c}_2t}{2}}(1+\tau)^{-4}\Big(1+\frac{(|y|-{\rm c}_2\tau)^2}{1+\tau}\Big)^{-2}\bigg(\frac{1+\tau}{1+t}\bigg)^Ndyd\tau\\
&\lesssim&\!\! (1+t)^{-2}\Big(1+\frac{(|x|-{\rm c}_1t)^2}{1+t}\Big)^{-1}\int_{0}^{\frac{t}{2}}\underbrace{(1+t-\tau)^{-1}(1+t-\tau)^{\frac{5}{2}\cdot\frac{2}{5}}(1+\tau)^{-3}(1+\tau)^{\frac{5}{2}\cdot \frac{3}{5}}}_{where\ we\ have\ used\ Young\ inequality\ with\ p=\frac{5}{2}\ and\ q=\frac{5}{3}}d\tau\\
&&+(1+t)^{-2}\Big(1+\frac{(|x|-{\rm c}_1t)^2}{1+t}\Big)^{-N}\int_{0}^{\frac{t}{2}}(1+\tau)^{-\frac{3}{2}}d\tau\\
&\lesssim&\!\!(1+t)^{-2}\Big(1+\frac{(|x|-{\rm c}_1t)^2}{1+t}\Big)^{-1}.
\end{eqnarray*}
On the other hand, the case $\frac{t}{2}\leq\tau\leq t$ can be estimated easily. We directly give the following: 
 \begin{eqnarray*}
J_{21} &\triangleq&\!\!\int_{\frac{t}{2}}^t\left( \int_{\big||y|-{\rm c}_2\tau\big|\geq\frac{|x|-{\rm c}_2t}{2}}+\int_{\big||y|-{\rm c}_2\tau\big|\leq\frac{|x|-{\rm c}_2t}{2}}\right) \left( \cdots \right) dyd\tau\\
&\lesssim&\!\!(1+t)^{-2}\Big(1+\frac{(|x|-{\rm c}_1t)^2}{1+t}\Big)^{-1}.
\end{eqnarray*}

In summary, combining the estimates $J_i$ $(i=1,2,\cdots,21)$ suffices to complete the proof of $\mathcal{K}_3$.  \ \ \ \ \ \ \ \ \ \ \ \ \ \ \ \ \ \ \ \ \textsquare

\vspace{3mm}
\textit{\textbf{Proof of $\mathcal{K}_1$ and $\mathcal{K}_2$ in Proposition \ref{l 4.2}}}}. Compared with the proof of $\mathcal{K}_3$, the proof of $\mathcal{K}_1$ and $\mathcal{K}_2$ are relatively straightforward. We just give the outline of the proof. At first,  we decompose space-time domain into the
following five regions:
\begin{align*}
\tilde{D}_{1}& =\left \{ |x|\leq \sqrt{1+t}\right \} \,, \\
\ \tilde{D}_{2}& =\left \{ {\rm c}t-\sqrt{1+t}\leq |x|\leq {\rm c}t+\sqrt{1+t}%
\right \} \,, \\
\tilde{D}_{3}& =\left \{ |x|\geq {\rm c}t+\sqrt{1+t}\right \} \,, \\
\tilde{D}_{4}& =\Big \{ \sqrt{1+t}\leq |x|\leq \frac{1}{2}{\rm c}t\Big \} \,,
\\
\tilde{D}_{5}& =\Big \{ \frac{1}{2}{\rm c}t\leq |x|\leq {\rm c}t-\sqrt{1+t}%
\Big \} \,,
\end{align*}
where ``${\rm c}$" can be chosen as the propagation speeds ${\rm c}_1$ and ${\rm c}_2$ of Huygens waves mentioned above. Second, by using Lemma \ref{A.7} and the similar decompositions of the temporal intervals in each space-time  domain $\tilde{D}_i$, one can obtain the desired space-time estimates in these domains. Finally, by combining these estimates, we finish the proof of $\mathcal{K}_1$ and $\mathcal{K}_2$. \ \ \ \ \ \ \ \ \ \ \ \ \ \ \ \ \ \ \ \ \ \ \ \ \ \ \ \ \ \ \ \ \ \ \ \ \ \ \ \ \ \ \ \ \ \ \ \ \ \ \ \ \ \ \ \ \ \ \ \ \ \ \ \ \ \ \ \ \ \ \ \ \ \ \textsquare

\section*{Acknowledgments}
Z.G. Wu was supported by Guangxi Natural Science Foundation (No. 2026GXNSFAA00641036) and National Natural Science Foundation of China (No. 12661043). Y.H. Zhang was spported by Guangxi Natural Science Foundation (No. 2026GXNSFFA00640002 and 2024GXNSFDA010071), National Natural Science Foundation of China (No. 12271114), Center for Applied Mathematics of Guangxi (Guangxi Normal University), and the Key Laboratory of Mathematical Model and Application (Guangxi Normal University), Education Department of Guangxi Zhuang Autonomous Region.

\vspace{6mm}

{\bf Data Availability:} No datasets were generated or analyzed during the current study.

\vspace{4mm}

{\bf Conflict of Interest:} The author declares that there is no conflict of interest.

\end{document}